\documentclass[aop,preprint]{imsart}
\setattribute{journal}{name}{}
\usepackage{mathtools,amsthm,amstext,amssymb,enumerate,array}
\usepackage{colonequals}
\usepackage{url,xspace}
\usepackage{graphicx,mathdots}
\usepackage[usenames,dvipsnames]{xcolor}
\usepackage{microtype}
\usepackage{tikz}
\usetikzlibrary{shapes}
\usetikzlibrary{arrows}
\usepackage[pdftitle={Finitary coloring},
            pdfpagelabels,
            pdfauthor={Alexander E. Holroyd, Oded Schramm, and David B. Wilson},
            colorlinks=true,linkcolor=NavyBlue,urlcolor=RoyalBlue,citecolor=PineGreen,
            bookmarks=true]{hyperref}
\usepackage[all]{hypcap}
\usepackage{cleveref}

\newcommand{\arXiv}[1]{\href{http://arxiv.org/abs/#1}{arXiv:#1}}

\newtheorem{theorem}{Theorem}
\newtheorem{proposition}[theorem]{Proposition}
\newtheorem{lemma}[theorem]{Lemma}
\newtheorem{corollary}[theorem]{Corollary}

\crefname{theorem}{Theorem}{Theorems}
\crefname{lemma}{Lemma}{Lemmas}
\crefname{proposition}{Proposition}{Propositions}
\crefname{corollary}{Corollary}{Corollaries}
\crefname{section}{Section}{Sections}
\crefname{figure}{Figure}{Figures}

\newcommand{\old}[1]{}

\newcommand{\Z}{{\mathbb Z}}
\newcommand{\R}{{\mathbb R}}

\renewcommand{\P}{{\mathbb P}}
\newcommand{\E}{{\mathbb E}}
\newcommand{\F}{{\mathcal F}}

\newcommand{\ind}{{\mathbf 1}}

\DeclareMathOperator{\tower}{tower}
\DeclareMathOperator{\esssup}{ess\; sup}
\DeclareMathOperator{\Var}{Var}
\DeclareMathOperator{\Cov}{Cov}
\DeclareMathOperator{\mmod}{mod}
\DeclareMathOperator{\rad}{rad}
\DeclareMathOperator{\diam}{diam}

\newcommand{\eps}{\varepsilon}

\newcommand{\tow}[4]{
{#1^{#2
\rule{0pt}{7pt}^{\iddots\rule{0pt}{9pt}
^{#3^{#4}}}}}
}

\newenvironment{ilist}{\begin{enumerate}[{\rm(i)}]}{\end{enumerate}}

\newcommand{\df}{\bf}

\begin{document}

\begin{frontmatter}
\title{Finitary Coloring}
\runtitle{Finitary Coloring}
\date{1 June 2015 (revised 20 April 2016)}
\let\MakeUppercase\relax 
\begin{aug}
\runauthor{Alexander E.~Holroyd, Oded Schramm, and David B.~\!Wilson}
\author{\href{http://research.microsoft.com/~holroyd}{Alexander E.~Holroyd}, \href{http://research.microsoft.com/~schramm}{Oded Schramm}, and \href{http://dbwilson.com}{David B.~\!Wilson}}
\address{Alexander E.~Holroyd, Microsoft Research,
One Microsoft Way, Redmond, WA 98052, USA}
\address{David B.~Wilson, Microsoft Research,
One Microsoft Way, Redmond, WA 98052, USA}
\affiliation{Microsoft Research}
\end{aug}

\setattribute{keyword}{AMS}{AMS 2010 subject classifications:}


\begin{abstract}
Suppose that the vertices of~$\Z^d$ are assigned random
colors via a finitary factor of independent identically
distributed (iid) vertex-labels.  That is, the color of
vertex~$v$ is determined by a rule that examines the
labels within a finite (but random and perhaps unbounded)
distance~$R$ of~$v$, and the same rule applies at all
vertices.  We investigate the tail behavior of~$R$ if the
coloring is required to be proper (that is, if adjacent
vertices must receive different colors).  When $d\geq 2$,
the optimal tail is given by a power law for $3$ colors,
and a tower (iterated exponential) function for $4$ or
more colors (and also for $3$ or more colors when $d=1$).
If proper coloring is replaced with any shift of finite
type in dimension $1$, then, apart from trivial cases,
tower function behavior also applies.
\end{abstract}
\end{frontmatter}
\maketitle

\section{Introduction}
A {\df $q$-coloring} of~$\Z^d$ is a random element $X{=}(X_v)_{v\in\Z^d}$ of
$\{1,\ldots,q\}^{\Z^d}$ that assigns distinct colors
 to neighboring sites; that is, almost surely $X_u\neq X_v$ whenever
$|u-v|=1$, where $|\cdot|$ is the $1$-norm on $\Z^d$.  We say that $X$ is a
{\df factor} of an iid process if it can be expressed as $X=F(Y)$ for some
family of iid random variables $Y=(Y_v)_{v\in\Z^d}$ and some measurable map
$F$ that is translation-equivariant (i.e.\ that commutes with the action of
every translation of~$\Z^d$).  We say that $X$ is a {\df finitary} factor of
an iid process, or simply that $X$ is {\df ffiid}, if furthermore, for almost
every $y$ (with respect to the law of~$Y$) there exists $r<\infty$ such that
whenever $y'$ agrees with $y$ on the ball $B(r)\colonequals\{v\in\Z^d:
|v|\leq r\}$, the resulting values assigned to the origin $0\in\Z^d$  agree,
i.e.\ $F(y')_0=F(y)_0$. In that case we write $R(y)$ for the minimum such
$r$, and we call the random variable $R=R(Y)$ the {\df coding radius} of the
factor.  In other words, in an ffiid coloring, the color at the origin can be
determined by examining the iid variables within distance given by the coding
radius (which is a finite but perhaps unbounded random variable).

We focus on the questions: for which $q$ and $d$ does an ffiid $q$-coloring
of~$\Z^d$ exist, and what can be said about the tail behavior of its coding
radius?  As a motivating example before stating our main results, we briefly
describe a simple construction of an ffiid $4$-coloring of~$\Z^2$ whose
coding radius has exponential tail decay; see \cref{example} for an
illustration.  Let $(B_v)_{v\in\Z^2}$ be iid labels taking values $+$ and $-$
with equal probabilities. Since the critical probability for site percolation
is greater than $\tfrac12$, almost surely all $(+)$-clusters and
$(-)$-clusters are finite.  Next we color each $(+)$-cluster with colors $1$
and $2$ in a checkerboard pattern. To ensure translation-equivariance, the
phase of the checkerboard must be chosen locally.  Here is one way to do
this. Assign color $1$ to the lexicographically largest site $w$ in the
$(+)$-cluster, and also to all other sites $v$ in the cluster for which the
sum of the coordinates of~$w-v$ is even; assign the remaining sites in the
cluster color $2$. Checkerboard the $(-)$-clusters with colors $3$ and $4$ in
the same manner. The resulting $4$-coloring is ffiid.  To determine the color
of the origin we must examine the labels~$B_v$ in its cluster and its
boundary.  Since the radius of the cluster has exponential tails, so does the
coding radius.
\begin{figure}
\includegraphics[width=0.5\textwidth]{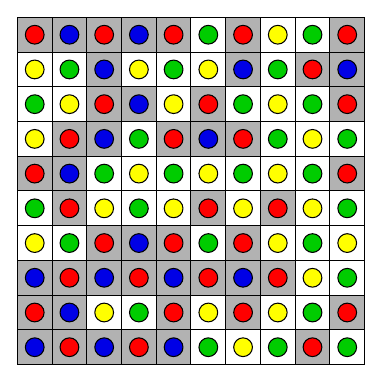}
\caption{An ffiid $4$-coloring of~$\Z^2$ whose coding radius has exponential tails.
Each (subcritical) site percolation cluster is assigned a checkerboard coloring.}\label{example}
\end{figure}

In fact, much faster decay than exponential is possible in many cases, while
only a power law is possible in others. For a non-negative integer~$r$,
define the {\df tower function} by $\tower(r)\colonequals \exp^r 1 =\exp \cdots \exp 1$,
where the exponential is iterated $r$ times.  For convenience we also write
$\tower(r)\colonequals\tower\lfloor r\rfloor$ for $r\in\R^+$.  Here are our main
results.
\begin{samepage}
\begin{theorem}[Tower function coloring] \label{tower}
Let $d=1$ and $q\geq 3$, or let $d\geq 2$ and $q\geq 4$.
There exist positive constants $c$ and $C$ depending on $q$ and $d$ such that the following hold.
\begin{ilist}
\item There exists an ffiid $q$-coloring of~$\Z^d$ whose
    coding radius satisfies
\[\P(R>r)< 1/\tower(cr),\qquad\forall r\geq 0\,.\]
\item Every ffiid $q$-coloring of~$\Z^d$ satisfies
\[\P(R>r)> 1/\tower(Cr),\qquad\forall r\geq 0\,.\]
\end{ilist}
\end{theorem}
\end{samepage}
\begin{theorem}[Power law $3$-coloring] \label{power}
Let $d\geq 2$.
\begin{ilist}
\item There exists a positive constant $\alpha$ (depending on $d$) and an ffiid $3$-coloring of~$\Z^d$ whose
    whose coding radius satisfies
\[\P(R>r)< r^{-\alpha},\qquad\forall r\geq 0\,.\]
\item Every ffiid $3$-coloring of~$\Z^d$ satisfies
\[\E (R^2) =\infty\,.\]
\end{ilist}
\end{theorem}
Since it is easy to see that no ffiid $2$-coloring of
$\Z^d$ is possible for any $d\geq 1$, \cref{tower,power}
determine the functional form (up to the various constants)
of the optimal tail decay of the coding radius for all~$q$
and $d$. Our proofs in principle yield explicit bounds on
the constants $c$, $C$ and $\alpha$, but $c$ and $C$ are
very far apart in most cases, while $\alpha$ is much
smaller than $2$.

\subsection*{Isometry equivariance}

We will prove that the colorings in the (i) parts of both
theorems can be chosen to have the stronger property that
the map $F$ from the iid variables to the coloring is
equivariant under all \textit{isometries\/} of~$\Z^d$. To
motivate this distinction, note that the percolation-based
construction of the $4$-coloring described above is not
isometry-equivariant, because using the lexicographic
ordering of~$\Z^2$ breaks rotation and reflection symmetry.
However, the construction can be modified as follows. Take
$(U_v)_{v\in\Z^2}$ iid uniform on $[0,1]$ and independent
of $(B_v)_{v\in \Z^2}$, and assign color $1$ to the site
$w$ in a $(+)$-cluster with the largest $U_w$, and to all
other sites of the same parity in the cluster (and
similarly for $(-)$-clusters). The resulting process is an
isometry-equivariant factor of the iid variables
$Y_v\colonequals(B_v,U_v)$, with the same coding radius as before.

\subsection*{Shifts of finite type}

Next we consider some generalizations, focussing on the
case $d=1$. Coloring is a special case of the more general
notion of a shift of finite type, in which the requirement
that adjacent colors differ is replaced with arbitrary
local constraints.  Write $[q]\colonequals\{1,\ldots,q\}$.
We call elements of $[q]^{\Z^d}$ {\df configurations}.  Let
$d=1$.  A {\df shift of finite type} is a (deterministic)
set of configurations $S$ characterized by an integer $k$
and a set $W\subseteq [q]^k$ of allowed local patterns as
follows:
\[S=S(q,k,W)\colonequals
\Bigl\{ x\in [q]^\Z: (x_{i+1},\ldots,x_{i+k})\in W\;\forall i\in \Z\Bigr\}\,.\]

We want to exclude a certain uninteresting case.  For $w\in
W$, let $T(w)$ be the set of times at which the pattern $w$
can recur, i.e.\ the set of~$t\geq 1$ for which there
exists $x\in S$ with $(x_1,\ldots,x_k)$ and
$(x_{t+1},\ldots,x_{t+k})$ both equal to $w$.  We call the
shift of finite type {\df non-lattice} if there exists
$w\in W$ for which $T(w)$ has greatest common divisor $1$
(and otherwise it is \textbf{lattice}).
If $S$ is non-lattice, then necessarily $S\neq\varnothing$.
For example, the set of all deterministic $q$-colorings
of~$\Z$ is a shift of finite type, and is non-lattice if
and only if $q\geq 3$.

\begin{theorem}[Shifts of finite type]\label{sft}
Let $S$ be a shift of finite type on~$\Z$.
\begin{ilist}
\item If $S$ is non-lattice then there exists an ffiid process $X$ such that
    $X\in S$ a.s., with coding radius $R$ satisfying \[\P(R>r)\leq
    1/\tower(cr),
    \qquad\forall r>0.\]
\item If $S$ contains no constant
    configuration
    $(\cdots aaa \cdots)$, then for any ffiid process $X$ such that $X\in S$ a.s.,
    the coding radius satisfies
    \[\P(R>r)\geq 1/\tower(Cr), \qquad\forall r>0.\]
\end{ilist}
Here $c,C$ are constants in $(0,\infty)$ that depend on $S$.
\end{theorem}

It is easily seen that for any \textit{lattice\/} shift of
finite type $S$, no ffiid process belongs a.s.\ to $S$.
Indeed, no \textit{mixing\/} process belongs to $S$ (see
\cref{no-mixing}). On the other hand, a constant
configuration is trivially an ffiid process with $R\equiv
0$.  Together with these obervations, \cref{sft} thus
covers all cases for $d=1$.

The concept of a shift of finite type extends in the
obvious way to $\Z^d$ (by requiring that the configuration
restricted to every ball of radius~$k$ lies in some fixed
set $W$).  For $d\geq 2$ we do not know what possible
restrictions on the coding radius can be imposed by the
requirement that an ffiid process belong to a given shift of
finite type, besides the possibilities already seen: tower
functions (e.g.\ $4$-coloring), power laws (e.g.\
$3$-coloring), and the two trivial cases of constant
sequences and lattice shifts of finite type.

\subsection*{Finite dependence}

Closely related to ffiid processes is the notion of
$k$-dependence. A process $X=(X_v)_{v\in\Z^d}$ on $\Z^d$ is
called {\df $k$-dependent} if $(X_v)_{v\in A}$ is
independent of $(X_v)_{v\in B}$ for any subsets
$A,B\subseteq \Z^d$ that satisfy $|u-v|>k$ for all $u\in A$
and $v\in B$.  A process is {\df finitely dependent} if it
is $k$-dependent for some $k$. A process $X$ is {\df
stationary} if $(X_{v+u})_{v\in \Z^d}$ and $(X_v)_{v\in
\Z^d}$ are equal in law for all $u$.  On the other hand,
$X$ is a {\df block factor} (of an iid process) if it is an
ffiid process with \textit{bounded\/} coding radius. When $d=1$ we say
that $X$ is a {\df $k$-block factor} if there exists iid
$Y$ and a fixed measurable function $g$ of~$k$ variables
such that $X_i=g(Y_{i+1},\ldots, Y_{i+k})$ a.s.\ for all
$i\in\Z$.

Clearly, any $k$-block factor on $\Z$ is stationary and
$(k-1)$-dependent. Much less obviously, the converse is
false.  This was an open question for some time (see e.g.\
\cite{janson-84}); the first counterexample appeared in
 \cite{aaronson-gilat-keane-devalk}.  Furthermore,
there exist $1$-dependent stationary processes that are not
$k$-block factors for any $k$; see
\cite{burton-goulet-meester}.  See e.g.\ \cite{hl} for more
on the history of this question, which apparently has its
origins in \cite{ibragimov-linnik}.

Since \cref{tower} (ii) implies that no $k$-block factor $q$-coloring exists
for any $k$ and $q$, it is natural to ask whether there is a stationary
$k$-dependent $q$-coloring.  It is easily seen that the answer is no if $k=0$
or if $q=2$.  We also establish a negative answer in the first nontrivial
case: $k=1$ and $q=3$.
\begin{theorem}
\label{fin-dep} There is no stationary $1$-dependent
$3$-coloring of~$\Z$.
\end{theorem}

Surprisingly, it has recently been proved \cite{hl} that
there exist both a stationary $1$-dependent $4$-coloring
and a stationary $2$-dependent $3$-coloring of~$\Z$.  Thus,
the above question is answered for all $k$ and $q$.
Moreover, coloring therefore provides a very clean and
natural proof of the non-equivalence of finitely dependent
processes and block factors. (Previous counterexamples have
tended to be somewhat contrived.)

By combining the $1$-dependent $4$-coloring of \cite{hl}
with results of the current article, it is also proved in
\cite{hl} that for all $d\geq 2$ there exists a stationary
$k$-dependent $4$-coloring of~$\Z^d$, for some $k=k(d)$,
and also that for any non-lattice shift of finite type $S$
on $\Z$ there exists a stationary $k$-dependent process
that lies in $S$ a.s., for some $k=k(S)$.

Combined with our \cref{sft} (ii), this last result
provides an even more striking illustration of the
difference between finitely dependent processes and block
factors: \textit{any\/} non-lattice shift of finite type with no
constant sequence serves to distinguish between them.

\enlargethispage*{1cm} The argument we use to prove
\cref{power} (ii) will also show (\cref{kdep-zd}) that no
stationary $k$-dependent $3$-coloring of~$\Z^d$ exists for
any $k$ and $d\geq 2$. See \cite{h-finitary},\cite{hl2} for
further recent work on $k$-dependent coloring.

\subsection*{Outline of proofs}

The existence of an ffiid $q$-coloring of~$\Z^d$ satisfying
a tower function bound with \textit{some\/} number of colors
$q=q(d)$ depends on a known method that was originally
motivated by applications in distributed computing.  The
method appeared first in \cite{CV}, and was developed
further in \cite{goldberg-plotkin-shannon}, \cite{linial}
and many subsequent articles.  The version that we use is
essentially that of \cite{linial}.

Translated to our setting and terminology, the method
mentioned above implies the existence of a \textit{block\/}
factor of an iid process that is ``almost'' a coloring, in
the sense that the probability of a violation (i.e.\ of two
given neighbors having the same color) is extremely small
as a function of the block radius. Such processes can be
constructed by starting with a discrete iid process and
iteratively applying an appropriate radius-$1$ block factor
that reduces the number of colors by a logarithmic function
without producing new violations.

In order to obtain an ffiid coloring we next proceed to
``stitch together'' an infinite family of the processes
described above, with different block radii and violation
probabilities.  This can be done even on a general graph of
bounded degree. In fact, the resulting factor satisfies a
much stronger property than automorphism equivariance: to
determine the color at a vertex, we do not need to know the
graph structure, except within the coding radius.

The most elaborate and novel part of the proof of
\cref{tower} (i) involves reduction of the number of colors
to $4$ in all dimensions $d\geq 2$. This is done by
applying carefully constructed block factors to colorings
with more colors, in order to obtain a $2$-valued process
with \textit{bounded\/} clusters.  After this, we conclude by
checkerboarding the clusters with two pairs of colors in
the manner mentioned earlier.  Many of the techniques in
this proof are quite general, and have wider applicability.
(One application appears in \cite{hl}.)

The tower function lower bound \cref{tower} (ii) is also a
consequence of a known result from distributed computing,
which was proved in \cite{naor}, building on earlier work
in \cite{linial0}.  We provide a proof that is arguably
simpler and more direct than the original proof.

Turning to \cref{power}, the proof of the existence of an
ffiid $3$-coloring with power law coding radius is
considerably simpler when $d=2$. The construction in this
case is based on critical bond percolation and its dual, on
a square lattice rotated by $45$ degrees.  We assign colors
to individual clusters based on their locations in a tree
structure arising from surrounding circuits. The power law
bound is a consequence of a Russo-Seymour-Welsh estimate.

The proof of \cref{power} (i) for general $d\geq 2$ is
broadly similar but more involved.  Instead of percolation
clusters, we use a partition of~$\Z^d$ that we construct
via an iterative scheme.  The sets of the partition are not
themselves independent sets, but contain pairs of
neighbors. Therefore, each set is assigned a checkerboard
coloring using $2$ of the available $3$ colors, and this
necessitates a more complicated tree argument. The method
is quite general, and can be extended to other graphs.

The second moment bound \cref{power} (ii) is a consequence
of the existence of a \textit{height function\/} for
$3$-colorings of~$\Z^2$.  The total height change around a
large contour must be zero, otherwise it is impossible to
extend the $3$-coloring to the interior.  However, if the
coding radius has finite second moment, the height changes
along distant parts of the contour are asymptotically
uncorrelated, leading to a contradiction.

Finally, the result on shifts of finite type is again
obtained from the result on tower function coloring by the
use of appropriate block factors, while the impossibility
of $1$-dependent $3$-coloring is proved by a conditioning
argument.

\section{Tower function lower bound}
\label{tower-bounds}

In this section we prove Theorem~\ref{tower} (ii).  The following is the key
fact.  An essentially equivalent result was proved in \cite{naor}, building
on earlier work of \cite{linial0}.  We give a simple direct proof.  Another
exposition and applications appear in \cite{alon-feldheim}. Recall that
$[q]\colonequals\{1,\ldots,q\}$.

\begin{proposition} \label{min}
  Let $(U_i)_{i\in\Z}$ be iid random variables
  taking values in an arbitrary set $B$, and let $r$ and $q$ be
  positive integers.
 For any measurable function $f:B^r\to[q]$,
\[\P\Big[ f(U_1,\ldots,U_r)=f(U_2,\ldots,U_{r+1}) \Big]
\geq
\frac{1}{\tow222{4q}}\,,\] where there are $r-1$
exponentiation operations in the tower.
\end{proposition}

\color{black}
If the $U_i$'s have a continuous distribution, then
\[\P\Big[ (U_1,\ldots,U_r)=(U_2,\ldots,U_{r+1}) \Big]=0\,,\] so it
\color{black}
is not obvious \textit{a priori\/} that the probability
in \cref{min} must be positive.  If the $U_i$'s have a discrete
distribution, the probability is positive, but it is not clear
\textit{a priori\/} that there is a positive lower bound depending
only on $r$ and $q$ that holds for all such distributions.
The results of Section~\ref{tower-general} below show that the tower
function bound is essentially tight.

\begin{proof}[Proof of Proposition~\ref{min}]
We will use induction on $r$.
  Let $\delta(r,q)$ be the largest number for which
  \[\P\Big[ f(U_1,\ldots,U_r)=f(U_2,\ldots,U_{r+1}) \Big]
  \geq\delta(r,q)\]
  for all choices of $B$, $f$, and the law of the $U_i$.
  When $r=1$ it is
  elementary that
  \[\P[ f(U_1)=f(U_2)] =\sum_{a=1}^q \P[f(U_1)=a]^2
  \geq \frac 1q\,,\]
 so $\delta(1,q)=1/q \geq 1/(4q)$, proving the result when $r=1$.

Now suppose $r\geq 2$.  Let $\eps\colonequals
\delta(r-1,2^q)/(2q)$, and
  define for $u_1,\ldots,u_{r-1}\in B$:

  \[S(u_1,\ldots,u_{r-1})\colonequals
  \Big\{a\in[q]:\P\big[f(u_1,\ldots,u_{r-1},U_r)=a\big]\geq
  \eps\Big\}\,.\]
  This is the set of values that $f$ assumes with
  probability $\geq \eps$ given the first $r-1$ arguments.  Since
  $S$ is a function on $B^{r-1}$ taking at most $2^q$ possible values
  (the subsets of~$[q]$), by the definition of~$\delta$ we have
\[\P\big[S(U_1,\ldots,U_{r-1})=S(U_2,\ldots,U_{r})\big]
\geq \delta(r-1,2^q)\,.\]
  But the definition of~$S$ implies
  \[\P\big[f(U_1,\ldots,U_{r})\notin S(U_1,\ldots, U_{r-1})\big]\leq
  q\eps\,,\]
  so we deduce
  \begin{equation}\label{d-e}
    \P\big[f(U_1,\ldots,U_{r})\in S(U_2,\ldots, U_r)\big]
    \geq\delta(r-1,2^q)-q\eps.
  \end{equation}

By the definition of~$S$ again, conditional on
$U_2,\ldots,U_r$, each element of $S(U_2,\ldots, U_r)$ has
probability at least $\eps$ as a possible value for the
random variable $f(U_2,\ldots,U_{r+1})$, and this remains
true if we condition also on $U_1$ (since $U_1$ and
$U_{r+1}$ are conditionally independent given
$U_2,\ldots,U_r$).  Therefore, almost surely
  \begin{multline*}
    \P\Big[ f(U_2,\ldots,U_{r+1})=f(U_1,\ldots,U_{r}) \;\Big|\;
    U_1,\ldots,U_r\Big] \\ \geq \ind\big[f(U_1,\ldots,U_{r})\in
    S(U_2,\ldots, U_r)\big]\times\eps.
  \end{multline*}
  Taking the expectation and using \eqref{d-e} gives
  \[\P\big[ f(U_1,\ldots,U_{r}){=}f(U_2,\ldots,U_{r+1})\big]
  \geq [\delta(r-1,2^q)-q\eps]\eps=\frac{\delta(r-1,2^q)^2}{4q}\,.\]
  Thus
\begin{equation}\label{ind-step}
  \delta(r,q) \geq \frac{\delta(r-1, 2^q)^2}{4q}.
\end{equation}

All that remains is to use \eqref{ind-step} to check the
claimed bound on $\delta$.  For $r=2$ we obtain
\[\delta(2,q) \geq \frac{1}{2^{2 q}} \frac{1}{4 q} \geq
\frac{1}{2^{4q}}\] as required.  We now use induction on
$r$ with base case $r=2$.  Since
  obviously $\delta(r,1)=1$, we assume $q\geq 2$.  Suppose
  $\delta(r,q) \geq 1 / {2^{2^{\cdots {2^{4 q}}}}}$ where there
are $r-1$ exponentiation operations in the tower.  Then
\eqref{ind-step}
  gives
\[ \delta(r+1,q) \geq
\frac1{\biggl(
\raisebox{-5pt}{$\tow222{4\times 2^q}$}
\biggr)^{\!\!2} 4q}
\geq \frac1{\biggl(
\raisebox{-5pt}{$\tow222{4\times 2^q}$}
\biggr)^{\!\!4}}
= \frac1{\tow{16}22{4\times 2^q}}\,.
\]
  Observe that when  $x\geq \frac{2}{3}$
 we have
$16^{2^{x}} = 2^{2^{x+2}} \leq 2^{2^{4 x}} = 2^{16^x}$, so
\[ \delta(r+1,q) \geq
\frac1{\tow22{16}{4\times 2^q}} \,.\] But $16^{4 \times
2^{q}} = 2^{2^{q+4}} \leq 2^{2^{4q}}$ for $q\geq 2$, which
completes the induction.
\end{proof}

The following notation will be useful. Suppose $X$ is an
ffiid process with underlying iid process $Y$ and coding
radius~$R$, and recall that $R=R(Y)$ where $R$ is a map
from configurations $y=(y_v)_{v\in\Z^d}$ to~$\Z$.  For
$v\in\Z^d$, define the coding radius at $v$ to be the
random variable
\[R_v\colonequals R(\theta_{-v} Y)\,,\] where
$\theta_{-v}$ denotes translation by $-v$, defined by
$(\theta_{-v}y)(u)\colonequals y_{u+v}$. Thus, $R_v$ is the radius
around $v$ up to which we need to examine the $Y$ variables
in order to determine $X_v$. Note that $R=R_0$, and that
the random variables $(R_v)_{v\in\Z^d}$ are identically
distributed.

\begin{proof}[Proof of Theorem \ref{tower} (ii)]
Let $X$ be an ffiid $q$-coloring of~$\Z^d$.  Suppose first
that $d=1$. Fix $r>0$ and define a modified process $X'$ by
\[X'_v \colonequals \begin{cases}
  X_v, &R_v\leq r;\\
  \infty, &R_v>r.
\end{cases}
\]
Then $X'$ is an ffiid process with coding radius bounded
above by $r$, i.e., $X'$ is a $(2r+1)$-block-factor. Since
$X$ is a coloring,
\[\P(X'_0=X'_1)=\P(X'_0=X'_1=\infty)=\P(R_0,R_1>r)\leq \P(R>r)\,.\]
On the other hand, Proposition~\ref{min} gives
\[\P(X'_0=X'_1)\geq
\frac 1 {\tow222{4(q+1)}}\,,\]
 with $2r$ exponentiations in the tower.
This is at least $1/\tower(Cr)$ for some $C$ depending only
on $q$, as required.

Now suppose $d\geq 2$.  The restriction of the coloring $X$
to the axis $\Z\times\{0\}^{d-1}$ is itself an ffiid
$q$-coloring of~$\Z$, with underlying iid process
$(Z_i)_{i\in\Z}$ given by the slices
$Z_i\colonequals (Y_{(i,w)})_{w\in\Z^{d-1}}$ (where $Y$ is the
underlying iid process for $X$). Furthermore, the coding
radius of the $1$-dimensional process is at most the coding
radius of~$X$, so the required bound follows from the
$1$-dimensional case proved above.
\end{proof}

\section{Tower coloring on general graphs}
\label{tower-general}

In preparation for the proof of \cref{tower} (i), in this
section we prove that on any graph of maximum degree
$\Delta$, there is an ffiid $(\Delta+1)$-coloring whose
coding radius has tower function tails, and that is an {\em
automorphism}-equivariant factor of the underlying iid
process. In particular, on $\Z^d$ this gives an
isometry-equivariant $(2d+1)$-coloring --- we will improve
this to $4$ colors for all $d\geq 2$ in the next section.

In fact, we will construct a coloring with a much stronger
property than automorphism-equivariance: the color at a
vertex can be determined locally without knowledge of the
graph itself --- we need only examine the iid labels and
the graph structure within the coding radius, and the
construction is invariant even under graph-automorphisms of
this local structure.  We now make this precise.

Let $G=(V,E)$ be a simple undirected graph.  We write $u\sim v$ if $\langle
u,v\rangle\in E$.  A {\df configuration} on $G$ is an element $z=(z_v)_{v\in
V}$ of~$\R^V$ that assigns labels to the vertices. A {\df labeled rooted
graph} is a triple $(G,o,z)$ consisting of a simple graph $G=(V,E)$, a root
$o\in V$, and a configuration $z$ on $G$. We call two labeled rooted graphs
\mbox{\df isomorphic} if there is a graph isomorphism between them that
preserves the root and the labels.  We call two labeled rooted graphs {\df
isomorphic to distance~$r$} if the labeled rooted subgraphs induced by the
respective sets of vertices within graph-distance~$r$ of their roots are
isomorphic.  A {\df local graph function} is a function $f$ from labeled
rooted graphs to $\R$, such that for every $(G,o,z)$ there exists $r\leq
\infty$ such that $f(G,o,z)=f(G',o',z')$ whenever $(G',o',z')$ and $(G,o,z)$
are isomorphic to distance~$r$. Let $R=R(f,G,o,z)$ be the minimum such~$r$.

A local graph function $f$ induces a map $F$ between
configurations on graphs as follows.  Let $G$ be a graph
and let $z$ be a configuration on $G$. Define the
configuration $F(z)$ by $(F(z))_v\colonequals f(G,v,z)$.  We call $F$
a {\df graph-factor map}.  A {\df process} on $G$ is a
random configuration $Z=(Z_v)_{v\in V}$, and it is {\df
$A$-valued} if each $Z_v$ takes values in a set
$A\subseteq\R$. If $Z$ is a process on $G$ and $X=F(Z)$
then we say that the process $X$ is a graph-factor of~$Z$,
and for $v\in V$ we call $R_v\colonequals R(f,G,v,Z)$ the {\df coding
radius} at $v$. If $R_v<\infty$ a.s.\ for all $v$ then it
is a {\df finitary} graph-factor, and if $R_v\leq r$ a.s.\
for all $v$ and some deterministic $r<\infty$ then it is a
{\df block} graph-factor. We call $X$ {\df graph-ffiid} if
it is a finitary graph-factor of some iid process. Recall
that $[q]\colonequals \{1,\ldots,q\}$.  A process~$X$ on a graph $G$
is a {\df $q$-coloring} if it is $[q]$-valued, and a.s.\
$X_u\neq X_v$ whenever $u \sim v$.

\begin{theorem}[Tower coloring on graphs]\label{graph}\sloppypar
Let $\Delta\geq 1$ be an integer.  There exists $C=C(\Delta)>0$ such that for
every graph $G$ of maximum degree $\Delta$, there is a graph-ffiid
$(\Delta+1)$-coloring of~$G$ such that for every vertex~$v$, the coding
radius~$R_v$ satisfies
\[\P(R_v>r)<1/\tower(C r), \qquad r>0.\]
\end{theorem}

The proof will actually give an even stronger fact: the
same local graph function may be used for all graphs of
maximum degree $\Delta$. The proof will proceed by
combining in a suitable way a family of block graph-factors
that are almost colorings in the sense that the probability
that neighbors share a color decays very rapidly as a
function of the block coding radius. As remarked earlier,
the existence of such block-factor processes is essentially
equivalent to known results in the distributed computing
literature.  However, the different focus in the latter
field makes it difficult to translate the results directly
into mathematical ones of the form we need. For the
reader's convenience we therefore provide a complete proof,
which is quite straightforward.

We will make extensive use of the fact that if $F$ and $G$
are block graph-factor maps with coding radii at most $r$
and $s$ then the composition $F\circ G$ is a block
graph-factor map with coding radius at most $r+s$.  We also
need the following simple result on set systems.
Refinements and generalizations appear in \cite{EFF}.

\begin{samepage}
\begin{lemma}[Set systems] \label{sperner}
\sloppy For each positive integer $d$ there exists $c=c(d)>0$ such that,
provided $n\leq e^{c k}$, there exists a family of~$n$ sets $S_1,\ldots
,S_n\subseteq [k]$ satisfying
\[S_{i_0}\not\subseteq S_{i_1}\cup\cdots\cup S_{i_d}\]
for all distinct $i_0,\ldots,i_d\in[n]$.
\end{lemma}
\end{samepage}

\begin{proof}
Let $S_1,\ldots,S_n$ be iid uniformly random subsets of
$[k]$. The probability that $S_{d+1}\subseteq
S_{1}\cup\cdots\cup S_{d}$ is $(1-2^{-d})^k=C^k$, say,
where $C=C(d)\in(0,1)$.  Therefore the expected number of
vectors $(i_0,\ldots,i_d)$ of distinct entries such that
$S_{i_0}\subseteq S_{i_1}\cup\cdots\cup S_{i_d}$ is at most
$n^{d+1} C^k$. This is strictly less than $1$ provided
$n<(C^{-1/(d+1)})^k$, which implies that there exist
families of sets for which there are no such vectors.
\end{proof}

We next prove the existence of ``almost colorings'' as
mentioned above.  Fix $\Delta\geq 1$. Let $c=c(\Delta)$ be
as in Lemma~\ref{sperner}, and define a sequence
$n_1<n_2<\cdots$ as follows. Let $n_1$ be the smallest
positive integer such that $\lfloor e^{cn_1}\rfloor>n_1$,
and define inductively for $i\geq 1$:
\[n_{i+1}\colonequals \lfloor e^{c n_i}\rfloor\,.\]
It is easy to check that $n_i \geq \tower(c'i)$ for all $i$
and some $c'=c'(\Delta)>0$. The following is a variant of a
result of \cite{linial}.

\begin{proposition}[Almost colorings]\label{tower-block}
Let $G=(V,E)$ be a graph of maximum degree $\Delta$, and
define $(n_k)_{k\geq 1}$ as above.  For each $k\geq 1$
there exists an $[n_1]\cup\{\infty\}$-valued block
graph-ffiid process $Y=Y^k$, with coding radius bounded
above by $k$ for every vertex, and with the following
properties. For adjacent vertices $u\sim v$ we have either
$Y_u\neq Y_v$ or $Y_u=\infty=Y_v$. For any vertex~$v$ we
have $\P(Y_v=\infty)\leq \Delta/n_k$.
\end{proposition}

\begin{proof}
We will construct a sequence of  processes
$Z^k,\ldots,Z^1$, each a radius-$1$ block graph-factor of
the previous one, ending with the required process $Y=Z^1$.
(The reverse indexing is a notational convenience.) The
process $Z^i$ will be $[n_i]\cup\{\infty\}$-valued. Let
$(Z_v)_{v\in V}$ be iid random variables, each uniform on
$[n_k]$. Define the first process $Z^k$ by setting
$Z^k_v\colonequals\infty$ if $Z_v=Z_u$ for some $u\sim v$, and
otherwise setting $Z^k_v\colonequals Z_v$.

Now suppose that $Z^k,\ldots,Z^{i+1}$ have been defined. We
will construct
 $Z^{i}$ from $Z^{i+1}$.  Fix a
family of~$n_{i+1}$ subsets $(S_j)_{j\in [n_{i+1}]}$
of~$[n_i]$ so that none is contained in the union of any
$\Delta$ others; Lemma~\ref{sperner} and the definition
of~$n_i$ ensure that this is possible.  For a vertex~$v$,
write $S(v)\colonequals S_{Z^{i+1}_v}$ for the
corresponding set, where we take $S_\infty\colonequals
\varnothing$.  Now define
\begin{equation}\label{reduction}
Z^{i}_v\colonequals \min\Bigl(S(v) \setminus \textstyle\bigcup_{u\sim v} S(u)\Bigr),
\end{equation}
where $\min\varnothing\colonequals\infty$.

We claim that for adjacent vertices $u \sim v$, and any
$i$, either $Z^i_u\neq Z^i_v$, or both are $\infty$, and
moreover, for any $v$ we have $Z^i_v=\infty$ if and only if
$Z^k_v=\infty$.  This follows easily by induction on $i$.
It certainly holds for $i=k$.  By \eqref{reduction} and
\cref{sperner}, if $Z^i_v=\infty$ then either
$Z^{i+1}_v=\infty$ or $Z^{i+1}_v=Z^{i+1}_u$ for some $u\sim
v$.  Moreover, for $u \sim v$, if $Z^{i}_u\neq\infty\neq
Z^{i}_v$ then $Z^{i}_v\in S(v)\setminus S(u)$ and
$Z^{i}_u\in S(u)\setminus S(v)$, so $Z^{i}_u\neq Z^{i}_v$.

Finally we set $Y=Z^1$.  It is evident from the
construction that $Y$ is a block graph-ffiid process with
coding radius at most $k$. We have
\[\P(Y_v=\infty)=\P(Z^k_v=\infty)=\P(Z_v=Z_u\text{ for some }u\sim v)\leq
\Delta/n_k. \qedhere\]
\end{proof}

In addition to the above result we will use the following
simple procedure for eliminating colors, which has other
applications also.  Let $\Z^+$ denote the positive
integers.  Suppose that $X$ is a $\Z^+
\cup\{\infty\}$-valued process on a graph $G=(V,E)$. Let
$a\in \Z^+$. We define a new process $\mathcal{E}_a X$ by
\[
(\mathcal{E}_a X)_v \colonequals
\begin{cases}
\min\Bigl(\Z^+ \setminus \{X_u: u\sim v\}\Bigr), & X_v=a;\\
X_v,& X_v\neq a.
\end{cases}
\]
Thus, the map $\mathcal{E}_a$ replaces color $a$ with the
smallest color that is absent from the neighbors of the
vertex.  This replacement color is in $[\Delta+1]$ if $G$
has maximum degree $\Delta$.  Neighboring vertices have
distinct colors in $\mathcal{E}_a X$ provided they do in
$X$.  Note that $\mathcal{E}_a$ is a radius-$1$ block
graph-factor map.

A simple application of the map defined above is that if
$X$ is a $q$-coloring of a graph of maximum degree
$\Delta$, then $\mathcal{E}_{\Delta+2}
\mathcal{E}_{\Delta+3} \cdots \mathcal{E}_{q} X$ is a
$(\Delta+1)$-coloring.  We use this idea in a more subtle
way in the next proof.

\begin{proof}[Proof of \cref{graph}]
Let $G=(V,E)$ be a graph of maximum degree $\Delta$.  Let
$(n_i)_{i\geq 1}$ be defined as above, and let
$Y^1,Y^2,\ldots$ be the processes of \cref{tower-block},
each constructed from the same iid family $(U_v)_{v\in V}$
(say by taking $Z_v=\lceil n_k U_v\rceil$ at the beginning
of the proof of \cref{tower-block}, where $U_v$ is uniform
on $[0,1]$). Recall that each $Y^k$ is
$[n_1]\cup\{\infty\}$-valued, and is a coloring except at
the vertices that are labeled $\infty$ (and that the
probability of label $\infty$ decreases rapidly with $k$).

We now construct a sequence of $[\Delta+1]\cup \{\infty\}$-valued processes
$X^0,X^1,X^2,\ldots$.  The desired coloring will be formed by taking their
limit.  First let $X^0_v\colonequals\infty$ for all $v$. Assuming
$X^0,\ldots,X^{k-1}$ have been defined, we next construct $X^{k}$ from
$X^{k-1}$ and $Y^k$. To do this, we first define an auxiliary
$[\Delta+1+n_1]\cup \{\infty\}$-valued process $W^{k}$ via
\[W^{k}_v\colonequals X^{k-1}_v \wedge (Y^{k}_v +\Delta+1)\,.\]
In other words, we construct $W^{k}$ from $X^{k-1}$ by
replacing occurrences of $\infty$ with the process $Y^k$
from the previous lemma, with the colors increased by
$\Delta+1$ so that they are distinct from the existing
ones. (Of course we take
$\infty+\Delta+1\colonequals\infty$). We now obtain $X^{k}$
from $W^{k}$ by eliminating these extra colors:
\[X^{k} \colonequals
\mathcal{E}_{\Delta+2} \mathcal{E}_{\Delta+3} \cdots
\mathcal{E}_{\Delta+1+n_1} W^{k}\,.
\]

Note that for any vertex~$v$, if $X^k_v\neq \infty$ for
some $k$ then $X^j_v$ is constant for all $j\geq k$. We
therefore define $X_v\colonequals\lim_{k\to\infty} X^k_v$.
By \cref{tower-block}, for all $k$,
\[\P(X_v=\infty)\leq\P(X^k_v=\infty)\leq \P(Y^k_v=\infty)\leq{\Delta}/{n_k}
\xrightarrow{k\to\infty} 0\,,\]
 and it follows that $X$ is a $(\Delta+1)$-coloring of~$G$.
Now, for any block graph-ffiid process $Z$, write $r(Z)$ for the smallest
constant $r$ such that the coding radius at every vertex is bounded above by
$r$. Then
\[r(X^k)\leq n_1+r(W^k)\leq n_1+[r(X^{k-1})\vee
r(Y^k)]=n_1+[r(X^{k-1})\vee k]\,.\]
 Hence we have $r(X^k)\leq n_1k+1$ for all $k$.
It follows that $X$ is graph-ffiid with coding radius~$R_v$
satisfying
\[\P(R_v>n_1k+1)\leq \P(X^k_v=\infty)\leq{\Delta}/{n_k}\]
for every $v$. As remarked earlier we have $n_i\leq
\tower(c'i)$ for some $c'=c'(\Delta)>0$, so the claimed
bound on $\P(R_v>r)$ follows.
\end{proof}

\section{Tower four-coloring}
In this section we prove \cref{tower} (i).  \cref{graph} in
the last section already gives an isometry-equivariant
ffiid $(2d+1)$-coloring of~$\Z^d$ for all $d\geq 1$, thus
proving the $d=1$ case.  For $d\geq 2$, the idea will be to
use \cref{graph} to obtain a coloring of a spread-out
lattice, and then apply carefully constructed block
factors.  We start by proving some more general results
that have applications elsewhere also.

We shift our focus back to processes on $\Z^d$.  A {\df
factor map} is a measurable map $F:\R^{\Z^d}\to\R^{\Z^d}$
between configurations that is translation-equivariant,
i.e.\ that commutes with the action of every translation of
$\Z^d$. Isometry-equivariance is defined analogously.  If
$X=F(Y)$ for a factor map $F$ then we say that $X$ is a
factor of~$Y$.  Finitary factors and coding radius are
defined as in the introduction.  A {\df block factor} map
is a finitary factor map whose coding radius is bounded
above, i.e.\ $R\leq k$ a.s.\ for some deterministic
$k<\infty$. Recall that $R_v\colonequals R\circ \theta^{-v}$ denotes
the coding radius at vertex~$v\in\Z^d$.

We say that a non-negative random variable $R$ has {\df
tower tails} if it satisfies $\P(R>r)<1/\tower(cr)$ for all
$r>0$ and some $c\in(0,\infty)$.  We call a process {\df
tower ffiid} if it is ffiid and its coding radius has tower
tails.  The following simple fact will be used extensively.
\begin{lemma}[Block factors]\label{block-of-tower}
If $X$ is a tower ffiid process on $\Z^d$ then any block
factor of~$X$ is tower ffiid.
\end{lemma}

\begin{proof}
Let $X$ be a tower factor of the iid process $Y$, and let
$W$ be a block factor of~$X$, with coding radius bounded
above by $k$.  Clearly $W$ is a factor of~$Y$. Write $R$
for the coding radius of~$X$, and as usual let $R_v$ be the
coding radius at $v\in\Z^d$. If $R'$ denotes the coding
radius of~$W$ viewed as a factor of~$Y$, then
\[\P(R'>r)\leq\P\Bigl[\bigcup_{v\in B(k)} \{R_v>r-k\}\Bigr]\leq
\frac{c_1 k^d}{\tower(c_2 (r-k))} \leq \frac{1}{\tower(c_3r)}\,,\]
for some
constants $c_i=c_i(k,d)\in(0,\infty)$.
\end{proof}

Let $\|\cdot\|_p$ denote the $p$-norm on $\Z^d$, and recall
that we usually work with the $1$-norm
$|\cdot|=\|\cdot\|_1$. For most purposes the distinction is
unimportant, because the norms are equivalent and we are
not concerned with exact constants. However, our
construction of a $4$-coloring will use both the $1$- and
$\infty$-norms.

A process $(X_v)_{v\in \Z^d}$ is a {\df range-$m$} {\df
$q$-coloring} with respect to the $p$-norm if it is
$[q]$-valued, and almost surely $X_u\neq X_v$ whenever
$0<\|u-v\|_p\leq m$.

\begin{corollary}[Long-range coloring]\label{range}\sloppypar
Fix integers $d,m\geq 1$ and a choice of norm
$\|\cdot\|_p$.  There exists a tower ffiid range-$m$
$q$-coloring of~$\Z^d$ with respect to $\|\cdot\|_p$, for
some number of colors $q=q(d,m,p)$.  Moreover, the factor
may be chosen to be isometry-equivariant.
\end{corollary}

\begin{proof}
This is a special case of \cref{graph}, applied to the
graph $\Z^d_{(m)}$ with vertex set $\Z^d$ and with an edge
between distinct $u,v\in\Z^d$ whenever $\|u-v\|_p\leq m$.
We can take $q\colonequals |\{v\in\Z^d: \|v\|_p\leq m\}|$.
\end{proof}

Let $m\geq 1$ be an integer.  A $\{0,1\}$-valued process
$J=(J_v)_{v\in\Z^d}$ is an {\df $m$-net} with respect to
the $p$-norm if a.s.\ for every vertex $u$ there exists $v$
with $\|u-v\|_p\leq m$ and $J(v)=1$, but there do not exist
distinct vertices $u,v$ with $\|u-v\|_p\leq m$ and
$J(u)=J(v)=1$. In other words, the support of~$J$ is a
maximal independent set in the graph $\Z^d_{(m)}$ defined
in the above proof. In dimension $d=1$, the distance
between any two consecutive $1$'s of an $m$-net lies in the
interval $[m+1,2m+1]$.

\begin{corollary}[Nets]\label{tower-net}
Fix integers $d,m\geq 1$ and a choice of norm
$\|\cdot\|_p$.  There exists a tower ffiid $m$-net on
$\Z^d$.   Moreover, the factor may be chosen to be
isometry-equivariant.
\end{corollary}

\begin{proof}
By \cref{range}, let $X$ be a tower-ffiid range-$m$
$q$-coloring.  Let $\mathcal{E}_a$ be the color-elimination
map defined in Section~\ref{tower-general}, for the graph
$\Z^d_{(m)}$ defined in the last proof. Recall that
$\mathcal{E}_a$ attempts to eliminate color $a$ by
replacing it with the smallest color that is absent from
the range-$m$ neighborhood of a vertex.  Now we attempt to
eliminate all colors:
\[Y\colonequals \mathcal{E}_1 \mathcal{E}_2 \cdots \mathcal{E}_q X\,.\]
The resulting process $Y$ is a coloring, and it is tower
ffiid by \cref{block-of-tower} (since $\mathcal{E}_a$ is a
block-factor map).  We claim that $J_v\colonequals
\ind[Y_v=1]$ yields the required $m$-net $J$. Indeed, $Y$
has no two $1$'s within distance $m$, while, if $X_v=a$
say, when we apply the map $\mathcal{E}_a$, the color at
$v$ becomes~$1$ provided there is currently no other $1$
within distance $m$ (and $1$'s remain $1$'s at subsequent
steps).
\end{proof}

In preparation for the proof of \cref{tower} (i) we record
the following simple geometric fact.

\begin{lemma}\label{geom}  Fix a norm.
Let $d\geq 1$ and let $c>0$ be a real constant. For any
$m\geq 1$ and any $m$-net $J$,  the number of $1$'s of~$J$
within distance $c\,m$ of any fixed $u\in\Z^d$ is at most
$C$, where $C$ is a constant depending only on $c$, $d$ and
the norm (not on $m$).
\end{lemma}

\begin{proof}
The balls of radius~$m/2$ centered at different $1$'s are
disjoint; consider their volumes.
\end{proof}

The next lemma enables a $4$-coloring of~$\Z^d$ to be
constructed from a $2$-valued process with bounded clusters
(via the checkerboard construction mentioned in the
introduction). As is customary, we denote by $\Z^d$ the
graph having vertex set $\Z^d$ and an edge between $u$ and $v$
whenever $\|u-v\|_1=1$. If $X$ is a process on $\Z^d$ then
an {\df $a$-cluster} of~$X$ is the vertex set of a
connected component of the subgraph of $\Z^d$ induced by
the (random) set of all $v$ with $X_v=a$. The {\df
diameter} (with respect to the $\infty$-norm) of a set
$A\subseteq \Z^d$ is $\sup\{\|u-v\|_\infty:u,v\in A\}$.

\begin{lemma}[Checkerboarding]\label{checker}
Fix integers $d,b\geq 1$.  Suppose $Y$ is a $[2]$-valued
process on $\Z^d$ in which each cluster has diameter at
most $b$ a.s.  There exists a $4$-coloring of $\Z^d$ that
is a block-factor of~$Y$.  Moreover, if $(U_v)_{v\in\Z^d}$
are iid, uniform on $[0,1]$ and independent of~$Y$, there
exists a $4$-coloring that is an isometry-equivariant
block-factor of the joint process $(Y,U)$.
\end{lemma}

\begin{proof}
We checkerboard each $1$-cluster with $1$'s and $3$'s, and
each $2$-cluster with $2$'s and $4$'s.  More formally, for
$v\in\Z^d$, let $w=w(v)\in\Z^d$ be the lexicographically
largest vertex in the same ($1$- or $2$-)cluster as~$v$.
(Or, for the isometry-equivariant version, let $w$ be the
vertex in the cluster for which $U_w$ is largest.)  Let
$X_v\colonequals Y_v+1+(-1)^{\|v-w\|_1}$. Then $X$ is a block-factor
of $Y$ because the clusters are bounded.
\end{proof}

Finally, our proof of \cref{tower} (i) will require the
following technical lemma.  A {\df slab} is a set of edges
of $\Z^d$ that is an image under some isometry of~$\Z^d$ of
the set
\[\Big\{\langle x,x+e_1\rangle: x\in\{0\}
\times \{1,\ldots, L\}^{d-1}\Big\}\,,\] for some $L>0$, where
$e_1=(1,0,\ldots,0)$ is the 1st coordinate vector. The slab
has {\df direction} $j\in\{1,\ldots,d\}$ if coordinate $j$
is the image of coordinate $1$ under the isometry.  By the
distance between two sets of edges we mean the
distance between their respective sets
of incident sites.

\enlargethispage*{1cm}
\begin{lemma}[Slabs]
\label{slabs} Suppose that $H$ is
 a subgraph of~$\Z^d$ whose edge
set is the union of a collection of slabs, such that no two
slabs of a given direction are within
$\|\cdot\|_\infty$-distance $2$. Each connected component
of~$H$ has $\|\cdot\|_\infty$-diameter at most $1$.
\end{lemma}

\begin{proof}
Consider the component of~$0$, and first consider edges in
direction~$1$. The given condition implies that for either
$s=0$ or $s=1$, all of the edges
\[\Bigl\{\langle x,x+e_1\rangle:x\in\{s,s-2\}
\times\{-1,0,1\}^{d-1}\Bigr\}\] are absent from $H$.  Since
similar statements hold for each coordinate, we deduce that
for some cube of $\|\cdot\|_\infty$-diameter $1$ containing
$0$, all the edges on the exterior boundary are absent from
$H$.
\end{proof}

By the {\df box} of radius~$r\in\Z$ centered at $v\in\Z^d$
we mean the $\infty$-norm ball
$\{u\in\Z^d:\|u-v\|_\infty\leq r\}$. The {\df boundary} of
a subset $A$ of~$\Z^d$ is the set of edges incident to a
site in $A$ and a site in $A^C$. The boundary of a box is a
union of a set of $2d$ slabs; we call them the {\df faces}
of the box.

\begin{proof}[Proof of \cref{tower} (i)]
As remarked earlier, the case $d=1$ and $q=3$ already
follows as a special case of \cref{graph}, therefore we
need to construct a $4$-coloring of $\Z^d$ for $d\geq 2$.
By \cref{block-of-tower,checker}, it suffices to construct
a tower ffiid $[2]$-valued process $Z$ with bounded
clusters.

Let $M=M(d)$ be a (large) positive integer to be fixed
later. By \cref{tower-net}, let $J$ be an $M$-net on $\Z^d$
with respect to $\|\cdot\|_\infty$, and let $S\colonequals
\{v\in\Z^d:J(v)=1\}$ be its support. Also, by \cref{range},
let $Y$ be a range-$(4M+3)$ $q$-coloring of $\Z^d$ with
respect to $\|\cdot\|_\infty$ (where we allow $q$ to be
chosen as a function of~$M$).  Take $J$ and $Y$ to be
finitary factors of the same iid process.  We will
construct a process~$Z$ with bounded clusters as a block
factor of~$(J,Y)$.  The coloring~$Y$ will appear in the
construction only in the form of its restriction to~$S$.
(In fact, an alternative variant of the proof would be to
instead use a coloring of the random graph with vertex set
$S$ and with an edge between elements at distance at most
$4M+3$, using \cref{graph}.)

We wish to assign an integer $r(s)\in [M,2M)$ to each
element $s$ of $S$ in such a way that, if we place a box of
radius~$r(s)$ centered at each $s\in S$, then no two faces
of a given direction are within
$\|\cdot\|_\infty$-distance~$2$ of each other. (So that we
can apply \cref{slabs}.)  This will be done iteratively in
the order given by the coloring $Y$.

Assuming radii have been chosen for all $s$ of colors $Y_s<j$ (which is
vacuously true when $j=1$), we will simultaneously choose a radius~$r(s)$ for
each $s\in S$ of color $Y_s=j$ in such a way that no faces of the box of
radius~$r(s)$ centered at $s$ come within distance $2$ of those faces already
chosen. By \cref{geom}, there are at most $C$ elements of~$S$ within
$\|\cdot\|_\infty$-distance $4M+2$ of~$s$, where $C$ is a constant that
depends only on $d$ (not on $M$, $j$, $q$, or $s$).  Any face of an existing
box centered at one of these elements prohibits at most $7$ possible values for
$r(s)$ in $[M,2M)$. Therefore, at most $C'\colonequals 14\, d\, C$ possible
values for $r(s)$ are prohibited by the condition on faces (in particular,
this $C'$ depends only on $d$). Also, since all radii are less than $2M$ but
$Y$ is a range-$(4M+3)$ coloring, the radii $r(s)$ for all those $s\in S$
with color $j$ can be chosen simultaneously without interfering with each
other (i.e.\ without two of them violating the face condition).  Therefore if
we choose $M=C'+1$ then these radii can indeed be chosen for each
$j=1,\ldots, q$ in turn. For definiteness and to ensure
isometry-equivariance, choose each $r(s)$ to be the smallest allowable value
in $[M,2M)$ at the appropriate step.

Now construct a $\{+1,-1\}$-valued process $Z$ as follows.
Any vertex~$v$ is covered by at least one of the boxes
chosen above (since $J$ is an $M$-net), but by only
finitely many.  Let $s=s(v)\in S$ be the center of the one
that has the lowest numbered color in $Y$. Let
$Z_v\colonequals  (-1)^{\|s-v\|_1}$. In other words, each
box is labeled checkerboard-fashion, with the parity
determined by the position of its center, and with
lower-colored boxes taking priority over higher ones. (We
are not using \cref{checker} here, despite the similarity
of the construction!)

Let $G$ be the (random) subgraph of $\Z^d$ in which two
adjacent vertices $u,v$ are connected by an edge if and
only if $Z_u=Z_v$. By the construction of the boxes, $G$ is
a subgraph of a graph $H$ satisfying the conditions of
\cref{slabs}, so each cluster of~$Z$ has
$\|\cdot\|_\infty$-diameter at most $1$, as required.

In each step $1,\ldots,q$ of the above procedure, a site
$s\in S$ only needed to examine $S$, $Y$, and the earlier
choices of radii within a neighborhood of radius~$4M+2$ in
order to determine its radius~$r(s)$. Thus the entire
procedure constitutes a block-factor map from $(J,Y)$ to
$Z$ (and indeed it is an isometry-equivariant map).
Therefore \cref{block-of-tower} gives that $Z$ is tower
ffiid.
\end{proof}

We note that the above argument actually gives the
following fact.  This has been used in \cite{hl} to prove
the existence of $k$-dependent $4$-colorings of~$\Z^d$ for
all $d\geq 2$.

\begin{corollary}
Let $d\geq 1$.  There exists $m$ such that for any $q$
there exists a block factor map $F$ with the following
property.  If $X$ is a range-$m$ $q$-coloring of~$\Z^d$
then $F(X)$ is a $4$-coloring of~$\Z^d$.  In addition, if
$(U_v)_{v\in\Z^d}$ are iid uniform on $[0,1]$ and
independent of~$X$ then similarly there exists an
isometry-equivariant block factor map $F'$ such that
$F'((X,U))$ is a $4$-coloring of~$\Z^d$.
\end{corollary}

\enlargethispage*{1cm}
\begin{proof}
We take $m=4M+3$ in the proof of \cref{tower} (i) above.
Since a range-$(4M+3)$ coloring is also a range-$M$
coloring, the construction in the proof of \cref{tower-net}
gives us an $M$-net $J$ as a block factor of~$X$, and we
also take $Y=X$.
\end{proof}

\section{Shifts of finite type}

In this section we prove \cref{sft}, for which we will use
the following construction.  Let $S=S(q,k,W)$ be a shift of
finite type on~$\Z$.  Let $G=G_S$ be the directed graph
with vertex set $W$, and with a directed edge from
$u=(u_1,\ldots,u_k)$ to $v=(v_{1},\ldots,v_k)$ if and only
if $(u_2,\ldots,u_k)=(v_{1},\ldots,v_{k-1})$.  For any
$x\in[q]^{\Z}$, clearly we have $x\in S$ if and only if the
sequence $\bigl((x_{i+1},\ldots,x_{i+k})\bigr)_{i\in\Z}$
forms a directed (bi-infinite) path in $G$.

\begin{proposition}[Shifts of finite type from nets]\label{net-sft}
Let $S$ be a non-lattice shift of finite type on~$\Z$.
There exist an integer $m\geq 1$ and a block-factor map $F$
such that if $J$ is an $m$-net then $F(J)$ belongs to $S$
a.s.
\end{proposition}

\begin{proof}
Let $S=S(q,k,W)$ and let $G=G_S$ be the directed graph
defined above. For $w\in W$, the set of recurrence times
$T(w)$ is precisely the set of positive integers $t$ for
which there exists a (not necessarily self-avoiding)
directed cycle of length $t$ in $G$ that contains the
vertex~$w$.  Suppose that the greatest common divisor
of~$T(w)$ is~$1$.   Since $T(w)$ is closed under addition,
it is a standard fact of number theory that there exists
some $m$ such that $T(w)$ contains all integers greater
than~$m$.

Therefore, for each integer $t\in[m+1,2m+1]$, we can fix a
directed cycle of~$G$ of length~$t$ containing~$w$. Let
$w=y^t_0,t^t_1,\ldots,y^t_t=w$ be its vertices in order.
Let $J$ be an $m$-net.  Construct a $W$-valued process $Z$
from $J$ as follows.  For each $i\in\Z$ with $J_i=1$, let
$Z_i=w$. If $i<j$ are the locations of two consecutive
$1$'s in $J$, let $t=j-i\in[m+1,2m+1]$, and let
$(Z_i,\ldots,Z_j)=(y^t_0,t^t_1,\ldots,y^t_t)$.  Finally,
define a process $X$ by letting $X_i$ be the first entry of
the $k$-vector $Z_i$ for each $i\in\Z$.  Clearly $X\in S$,
and $X$ is a block factor of~$J$ because the intervals
between $1$'s of~$J$ have bounded lengths.
\end{proof}

\begin{proof}[Proof of \cref{sft} (i)]\sloppypar
This follows immediately from \cref{tower-net,net-sft}.
\end{proof}

We note that our argument yields the following, which is
used in \cite{hl}.

\begin{corollary}
Let $S$ be a non-lattice shift of finite type on~$\Z$.
There exist $m$ such that for any $q$, there exists a
block-factor map $F$ such that if $X$ is a range-$m$
$q$-coloring of~$\Z$ then $F(J)$ belongs to $S$.
\end{corollary}

\begin{proof}
This follows from \cref{net-sft} and the proof of \cref{tower-net}.
\end{proof}

\begin{proof}[Proof of \cref{sft} (ii)]
Suppose that $S=S(q,k,W)$.  If $S$ contains no constant
sequence then the graph $G=G_S$ has no self-loops.  Suppose
$X$ is an ffiid process that belongs to $S$ a.s.  Then the
block process $W=(W_i)_{i\in\Z}$ given by $W_i\colonequals
(X_{i+1},\ldots,X_{i+k})$ is a $q^k$-coloring of $\Z$, and
it is clearly a block factor of~$X$.  Let $R$ be the coding
radius of~$X$, and let $R'$ be the coding radius of $W$
viewed as a factor of the iid process underlying~$X$.
Theorem~\ref{tower} (ii) implies $\P(R'>r)\geq
1/\tower(Cr)$ for all $r$ and some $c$, while as in the
proof of Lemma~\ref{block-of-tower}, \mbox{$\P(R'>r)\leq
k\,\P(R>r-k)$}. Hence $\P(R>r)\geq 1/\tower(C'r)$ for some
$C'=C'(C,k)$.
\end{proof}

Finally in this section we show that a \textit{lattice\/} shift
of finite type admits no ffiid process, as mentioned in the
introduction.  In fact we prove a stronger statement.  A
process $X$ on $\Z$ is called {\df mixing} if for any
events $A$ and $B$ in the $\sigma$-field generated by $X$
we have $\P(A\cap\theta^n B)\to\P(A)\P(B)$ as $n\to\infty$.
(Here, if $A$ is the event $\{X\in \mathcal A\}$ then
$\theta^n A$ is the translated event $\{(X_{i+n})_{i\in
\Z}\in\mathcal A\}$.) The following is a standard fact.

\begin{lemma}\label{mixing}
Any factor of an iid process on $\Z$ is mixing.
\end{lemma}

\begin{proof}
Suppose $X$ is a factor of the iid process $Y$.  Fix events $A,B\in\sigma(X)$
and any $\eps>0$.  There exist cylinder events $A_\eps,B_\eps$ of~$Y$ such
that $\P(A\triangle A_\eps),\P(B\triangle B_\eps)<\eps$, and by
translation-equivariance, $\P(\theta^n B\triangle \theta^n B_\eps)<\eps$. For
$n$ sufficiently large, $A_\eps$ and $\theta^n B_\eps$ are independent, and
hence $|\P(A\cap\theta^n B)-\P(A)\P(B)|<4\eps$.
\end{proof}

\begin{proposition} \label{no-mixing}
Let $S$ be a lattice shift of finite type on~$\Z$.  There is no mixing
stationary process $X$ for which $X\in S$ a.s.
\end{proposition}

\begin{proof}
Suppose that such an $X$ does exist.  Since $X$ is mixing, it is ergodic.
Hence there exists some $w\in W$ that a.s.\ appears infinitely often in the
process $W$ given by $W_i\colonequals (X_{i+1},\ldots,X_{i+k})$. Fix such a $w$, and let
$t$ be the greatest common divisor of the recurrence set $T(w)$.  Then a.s.\
the random set $\{i\in\Z: W_i=w\}$ lies in $L+t\Z$ for some random $L$ in
$[t]$. Since the set is a.s.\ non-empty, $L$ is measurable with respect to
$\sigma(X)$, and by stationarity $L$ must be uniformly distributed over
$[t]$.  Therefore, letting $A$ be the event that $L=t$, we have $\P(A \cap
\theta^n A)= \ind[t \text{ divides }n]/t$, which does not converge as
$n\to\infty$, contradicting the fact that $X$ is mixing.
\end{proof}

\section{Power law coloring}

In this section we construct ffiid $3$-colorings of~$\Z^d$
for $d\geq 2$ with power law tails, proving
Theorem~\ref{power} (i).  A simpler version of the argument
is available when $d=2$; we give this first.

\begin{proof}[Proof of Theorem~\ref{power} (i), case $d=2$]
First construct a random graph $H$ with vertex set $\Z^2$
by choosing, for each unit square of~$\Z^2$, exactly one of
the two diagonals to be an edge of $H$, with each diagonal
having probability $1/2$, and where the choices are
independent for different squares.  It is of course trivial
to do this as a translation-equivariant block factor of an
iid process indexed by the vertices.  For an
\textit{isometry\/}-equivariant construction one can
proceed as follows.  Let $(U_v)_{v\in\Z^2}$  be iid uniform
on $[0,1]$ and let $(B_v)_{v\in \Z^2}$ be iid uniform on
$\{\pm 1\}$, independent of each other. For a unit square
$s$ define $B'_s\colonequals \prod_{i=1}^4 B_{s_i}$, where
$s_1,\ldots,s_4$ are the vertices (in counterclockwise
order, say).  Then $(B'_s)_s$ is an iid uniform $\pm
1$-valued family indexed by unit squares (as can be seen by
considering in lexicographic order the unit squares that
make up an $n$ by $n$ square, and noting that each is
independent of those preceding~it). Now place an edge
between $s_1,s_3$ if
$(U_{s_1}+U_{s_3}-U_{s_2}-U_{s_4})\,B_s>0$, and otherwise
place it between $s_2,s_4$.

Observe that $H$ is precisely a critical bond percolation
model on the even sublattice of~$\Z^2$ (interpreted as a
copy of~$\Z^2$ rotated by $\pi/4$ and enlarged by
$\sqrt{2}$) together with its planar dual on the odd
sublattice.  See \cref{perc}.  Note that for the purpose of
constructing an ffiid process, it is important that we
treat the even and odd sublattices identically.

\begin{figure}
\includegraphics[width=0.65\textwidth]{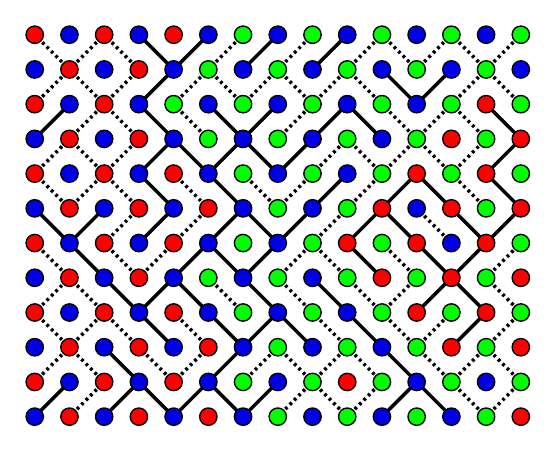}
\caption{Random diagonals, the resulting bond percolation process (solid lines),
its planar dual (dashed lines),
and a corresponding $3$-coloring.}\label{perc}
\end{figure}

Call the connected components of~$H$ {\df clusters}, and
call two clusters {\df adjacent} if some vertex of one is
adjacent in $\Z^2$ to some vertex of the other.  (Adjacent
clusters belong to sublattices of opposite parity, of
course.)  We will assign one of the $3$ colors to each
cluster. This will result in a coloring of~$\Z^2$ provided
adjacent clusters receive distinct colors, as in
\cref{perc}.

All clusters are finite a.s.\ (since there is no
percolation at the critical point $1/2$ of bond percolation
on $\Z^2$ --- see e.g.\ \cite{grimmett}).  For each
cluster~$K$ there is precisely one cluster~$\pi(K)$ that
surrounds $K$ (i.e.\ intersects every infinite path from
$K$) and is adjacent to $K$ (see e.g.\ \cite{grimmett}). We
call $\pi(K)$ the {\df parent} of~$K$, and $K$ a {\df
child} of~$\pi(K)$.  Any two adjacent clusters are parent
and child in exactly one direction.  If $K'=\pi^m(K)$ for
some $m\geq 0$ (where $\pi^m$ denotes the $m$th iterate of
$\pi$) then we say that $K'$ is an {\df ancestor} of~$K$
and that $K$ is a {\df descendant} of~$K'$. Note that each
cluster has infinitely many ancestors but only finitely
many descendants.

Next we assign a label $Y_K$ to each cluster~$K$, in such a
way that conditional on $H$ the labels are iid and uniform
on $\{\pm1\}$. To do this, take $(V_v)_{v\in\Z^2}$ iid
uniform on $[0,1]$ and $(W_v)_{v\in\Z^2}$ iid uniform on
$\{\pm 1\}$, and let $Y_K=W_u$ where $u$ is the vertex
of~$K$ for which $V_u$ is greatest.  Call a cluster~$K$
{\df special} if $Y_K=1$ but $Y_{\pi(K)}=-1$.  Now we
define the coloring.  Assign color $1$ to each special
cluster.  For a non-special cluster~$K$, let $m\geq 1$ be
the smallest positive integer for which the ancestor
$\pi^m(K)$ is special, and assign $K$ color $2$ or $3$
according to whether $m$ is odd or even respectively.

The above clearly gives a coloring.   To check that it is
ffiid and bound the coding radius, note that to determine
the color of the origin, it suffices to examine the various
iid labels of the parent of the most recent special
ancestor of the cluster of the origin, together with those
of all its descendants, and the vertices of~$\Z^2$ within
distance $2$ of these clusters. The coding radius~$R$ is at
most the radius around $0$ of this set of vertices. To
bound $R$, define a family of nested annuli
$A_n\colonequals \{x\in\Z^2: 2^n\leq |x|<2^{n+1}\}$
centered at the origin.  By the Russo-Seymour-Welsh
theorem, the probability that $H$ contains a circuit in the
even sublattice that lies in $A_n$ and surrounds the origin
is bounded strictly away from $0$ as $n\to\infty$, and
similarly for the odd sublattice --- see e.g.\
\cite{grimmett}. Take $p>0$ and $N\geq 1$ such that both
probabilities are at least $p$ for all $n> N$.  Now let
$E_m$ be the event that the following all hold: $A_{4m}$
and $A_{4m+2}$ each contain such a circuit in the even
sublattice, while $A_{4m+1}$ and $A_{4m+3}$ each contain
one in the odd sublattice, and moreover, the cluster that
contains the outermost such circuit in $A_{4m+1}$ is
special.  Now the events $(E_m)_{m\geq 1}$ are independent,
and $\P(E_m)>p^4/4$ if $4m>N$.  If $E_m$ occurs then the
cluster of the origin has a special ancestor whose parent
lies within the ball $B(2^{4m+4})$. Therefore
$\P(R>2^{4m+4}+2)\leq (1-p^4/4)^m$ for $4m>N$, which gives
the claimed power law tail bound.
\end{proof}

Unfortunately, the above method gives only a very small
positive power $\alpha$ in the bound
$\P(R>r)<c\,r^{-\alpha}$.  The best available lower bound
for the Russo-Seymour-Welsh circuit probability $p$ is
roughly $2^{-36}$.   And, even with more elaborate
bookkeeping, the best that can be obtained from the above
argument is $\P(R>2^{m})\leq (1-p/2)^m$, giving
$\alpha\approx p/(2\log 2)$. It would be of interest to
obtain a more reasonable power (either for this
$3$-coloring of~$\Z^2$ or another one).

We now move on to the case of general $d\geq 2$.  The
strategy will be broadly similar to that for $d=2$ above,
but with the following main differences.  We can no longer
use critical percolation together with its planar dual;
instead, we use an iterative procedure to construct a
partition of~$\Z^d$ with a similar tree structure.
However, unlike the percolation clusters, individual sets
of this partition will themselves contain pairs of
neighboring vertices. Therefore, rather than a single
color, each set will be assigned a checkerboard
$2$-coloring comprising $2$ of the $3$ available colors.
This in turn will necessitate a more subtle version of the
family tree coloring procedure. The method of proof is
quite general, and can be applied to other graphs (with an
appropriate number of colors that depends on the graph).

The first part of the construction is deterministic, and
can be done on any graph. (In fact, it can be generalized
to metric spaces.)  Let $G=(V,E)$ be a simple undirected
graph, and let $\delta$ denote graph-distance on~$V$.
Denote the closed ball $B(u,r)\colonequals \{v\in V: \delta(u,v)\leq
r\}$.  As usual, the diameter of a set $S\subseteq V$ is
$\diam (S)\colonequals \sup\{\delta(u,v):u,v\in S\}$, the radius
around a point $u\in S$ is
$\rad_u(S)\colonequals\sup\{\delta(u,v):v\in S\}$, and the (graph)
distance between two sets $S,T\subseteq V$ is
$\delta(S,T)\colonequals\inf\{\delta(s,t):s\in S,\, t\in T\}$.

Here is the construction.  Define \[r_j\colonequals 13^j,\qquad j\geq
1\] and suppose we are given a family of sets
$V_1,V_2,\ldots \subseteq V$.  (In our application below,
the sets will be chosen randomly, in such a way that no two
elements of~$V_j$ are within distance $4\, r_j$ of each
other.) We call elements of~$V_j$ {\df $j$-centers}.  Call
the ball of radius~$r_j$ centered at any $j$-center a {\df
$j$-ball}.  To each $j$-ball we will associate a subset of
$V$, called a {\df $j$-tile}.  The collection of all tiles
will be our partition.

The $1$-tiles are precisely the $1$-balls.  Now assume that
$j$-tiles have been defined for all $j\leq n$, and let
$\mathcal{T}_n$ denote the set of all such tiles.  Let
$\mathcal{G}_n$ be the graph with vertex set
$\mathcal{T}_n$ in which two tiles are neighbors in
$\mathcal{G}_n$ if the distance between them is at most
$2$.  Define an {\df $n$-clump} to be the union of the
tiles that correspond to a connected component of~$\mathcal{G}_n$.
By the $n$-clump of a tile we mean the $n$-clump containing that tile.

Now let $B$ be an $(n+1)$-ball. Let $S_B$ denote the union
of $B$ and all the $n$-clumps that are within distance at
most $2$ from $B$. Define the $(n+1)$-tile $T_B$ to be the
set of all $v\in V$ that are within distance at most $1$
from $S_B$ but are not in $\bigcup \mathcal{T}_n$.  The
$(n+1)$-tiles are all such $T_B$.

At the same time as defining tiles, we impose a family tree
structure on them. Every tile $T'$ of~$\mathcal{T}_n$ that
is a subset of~$S_B$ is declared a {\df child} of~$T_B$,
provided $T'$ was not already declared a child of some
other tile at some earlier stage.  If $T'$ is a child of
$T$ then $T$ is a {\df parent} of~$T'$. \textit{A~priori\/}
a tile might have no parents, or more than one, but we will
see next that for our choice of $V_j$'s the parent is
unique.

Each tile has a {\df center},
defined to be the center of the ball $B$ used to define the
tile $T_B$. (The center is not necessarily an element of
the tile.)

\begin{lemma}[Tiling]\label{tiling}
Let $G=(V,E)$ be an infinite connected graph, let $V_1,V_2,\ldots \subseteq
V$ be sets of centers, and construct tiles as described above.  Suppose that
every $v\in V$ lies in some ball, and that no two $j$-centers are within
distance $4 \,r_j$ (for each $j\geq 1$).  Then the set of all tiles is a
partition of~$V$.  Each tile is non-empty, and has exactly one parent. If
$T,T'$ are distinct tiles neither of which is a child of the other then
$\delta(T,T')>1$. If there is a $j$-tile centered at $v$, then the tile and
its associated $j$-clump are subsets of the ball $B(v,\frac32\, r_j)$, and
are functions of $V_1,\dots,V_j$ restricted to this ball.
\end{lemma}

\begin{proof}
The key step is to prove by induction that the diameter of a $j$-clump is at
most $3\, r_j$. This certainly holds for $j=1$.  Assume that it holds for
$j=n$. Let $B$ be an $(n+1)$-ball with center $u$.  Recalling the definition
of the associated tile $T_B$, we can bound its radius:
\[\rad_u(T_B) \leq r_{n+1} + 2 + 3\, r_n + 1\,.\]
Let $\widehat T_B$ be the union of $T_B$  with all the
$n$-clumps that are within distance at most $2$ from $T_B$.
Then
\[\rad_u(\widehat T_B) \leq \rad_u(T_B) + 2 + 3\, r_n \leq r_{n+1} + 6\, r_n
+5
 < \tfrac 32 \, r_{n+1}\,,\]
by our choice of~$r_j$.  For distinct $(n+1)$-balls
$B_1,B_2$, the centers are at distance at least $4\,
r_{n+1}$, therefore $\delta(\widehat T_{B_1},\widehat
T_{B_2})> (4-2\cdot\tfrac32)r_{n+1}=r_{n+1}>2$. It follows
that the $(n+1)$-clump of~$T_B$ is $\widehat T_B$, and
hence that this clump has diameter at most  $3\, r_{n+1}$.
This completes the induction.

From the above inequality, in fact the radius of $T_B$'s clump $\widehat T_B$
is at most $\tfrac32\, r_{n+1}$, and by the construction of $T_B$, the tile
and the clump are functions of $V_1,\dots,V_j$ restricted to the ball of
radius $\tfrac32\,r_{n+1}$ centered at $u$,
 as claimed.

Now, if $v$ lies in an $n$-ball $B$ then either $v$ lies in
$T_B$, or it lies in some tile of~$\mathcal{T}_{n-1}$. Thus
every $v$ lies in some tile. On the other hand, we showed
above that any two $n$-tiles are disjoint (and in fact are
at distance greater than $2$), while by the construction,
an $n$-tile is disjoint from $\bigcup\mathcal{T}_{n-1}$.
Thus the tiles partition~$V$.

To see that the $n$-tile $T_B$ is non-empty, recall that
$T_B\supseteq B\setminus \bigcup \mathcal{T}_{n-1}$.  But
we cannot have $\bigcup \mathcal{T}_{n-1}\supseteq B$,
because $B$ is connected, while each component of $\bigcup
\mathcal{T}_{n-1}$ lies in an $(n-1)$-clump, and thus has
strictly smaller diameter than $B$.

Let $T,T'$ be distinct tiles neither of which is a child of
the other.  As remarked above, if both are $n$-tiles then
$\delta(T,T')>2>1$.  On the other hand, if $T=T_B$ is an
$n$-tile and $T'\in\mathcal{T}_{n-1}$ then, by the
definition of~$S_B$, either $T'\subseteq S_B$ or
$\delta(T',S_B)>2$.  In the former case, $T'$ was already
assigned a parent before stage $n$, and thus all vertices
of $V$ that are within distance $1$ of $T'$ lie in $\bigcup
\mathcal{T}_{n-1}$, so $\delta(T,T')>1$.  In the latter
case, the definition of~$T_B$ implies that $\delta(T,T')>1$
also.

 If $B$ is an $n$-ball, then $S_B$ is contained in the
$n$-clump of $T_B$, but we showed above that for distinct
$n$-balls $B_1$ and $B_2$, the clumps of $T_{B_1}$ and
$T_{B_2}$ are disjoint.  Thus any tile has at most one
parent.  It remains to show that an $n$-tile $T$ has at
least one parent. Since $G$ is infinite and connected but
the $n$-clump of~$T$ is bounded, there exists $w\in V$ that
is at distance $1$ from the clump but not in~$\bigcup
\mathcal{T}_n$.  This $w$ lies in some ball $B$, which must
be a $m$-ball for some $m>n$ (otherwise $w$ would lie in
$\bigcup \mathcal{T}_n$), and thus $S_B$ contains $T$.
Hence either $T_B$ is the parent of~$T$, or another tile
was declared the parent of~$T$ before $T_B$ was
constructed.
\end{proof}

As before, we write $\pi(T)$ for the parent of a tile $T$.
If $T'=\pi^m(T)$ for some $m\geq 0$ then we call $T'$ an
{\df ancestor} of~$T$, and $T$ a {\df descendant} of~$T'$.
Let $\mathcal{F}$ denote the graph whose vertices are the
tiles, and where two tiles are adjacent if they are at
distance $1$. Thus $\mathcal{F}$ is a forest with exactly
one end per component.  (In our application below,
$\mathcal{F}$ will actually be a tree.)

In order to bound the coding radius of our coloring, we
need the following additional property.

\begin{lemma}\label{descendant}
Assume the conditions of \cref{tiling}.  If the ball $B$
contains the vertex $v$ then some descendant of tile~$T_B$
contains the vertex~$v$.
\end{lemma}

\begin{proof}
Suppose $B$ is an $n$-ball. Either $v$ lies in $T_B$ itself, or it lies in
some previously constructed tile $T$ that is a subset of~$S_B$. In the latter
case, either $T$ is a child of~$T_B$, or $T$ was earlier declared a child of
some other tile $\pi(T)=T_{B'}$, say, where $B'$ is an $n'$-ball and $n'<n$.
In that case, $S_{B'}$ contains $T$ (by the definition of child). Since
the $n'$-clump of~$T$ is a subset of the $n$-clump of $T$ (by the
definition of clump), we have that $S_{B'}\subseteq S_B$ and therefore
$\pi(T)\subseteq S_B$. Now we iterate this argument: the parent $\pi^2(T)$
of~$\pi(T)$ is either $T_B$, or it is some other tile constructed after
$\pi(T)$ but before $T_B$, in which case $\pi^2(T)\subseteq S_B$, and so on.
Eventually we conclude that $T_B$ is an ancestor of~$T$.
\end{proof}

\begin{proof}[Proof of Theorem~\ref{power} (i) for general $d\geq 2$]\
We first construct a random tiling of~$\Z^d$.  Define $r_j=13^j$ as above.
For $j\geq 1$, let $W_j$ be a random subset of~$\Z^d$ in which each vertex is
included with probability $r_j^{-d}$, independently for different vertices.
Let $V_j$ be the set of elements of~$W_j$ that have no other element of~$W_j$
within distance $4\, r_j$. Let the sets $(V_j)_{j\geq 1}$ be independent of
each other.  Construct tiles using the sets of centers $(V_j)_{j\geq 1}$ as
described above.  Note that the probability that $v\in\Z^d$ lies in some
$j$-ball is at least $\eta$ for some $\eta=\eta(d)>0$ that does not depend on
$j$.  Therefore, every $v$ lies in some ball, so \cref{tiling} applies.

Let $Q$ be the set of all deterministic colorings of $\Z^d$
that use any $2$ colors from $\{1,2,3\}$. Then $Q$ has $6$
elements, since there are $\binom{3}{2}$ choices of $2$
colors, and $2$ possible checkerboard phases. Consider the
graph $\mathcal{Q}$ with vertex set $Q$, and with an edge
between two colorings if one can be obtained from the other
by exchanging one color for the unused color, together with
a self-loop at each vertex.  Thus $\mathcal{Q}$ is a
hexagon with self-loops, and hence has diameter
$D\colonequals 3$. See \cref{hexagon}. We will assign a
coloring in $Q$ to each tile, and this will result in a
coloring of~$\Z^d$ provided adjacent tiles receive
colorings that are adjacent in $\mathcal{Q}$.  For every
pair of vertices of~$\mathcal{Q}$, fix a canonical shortest
path between them.
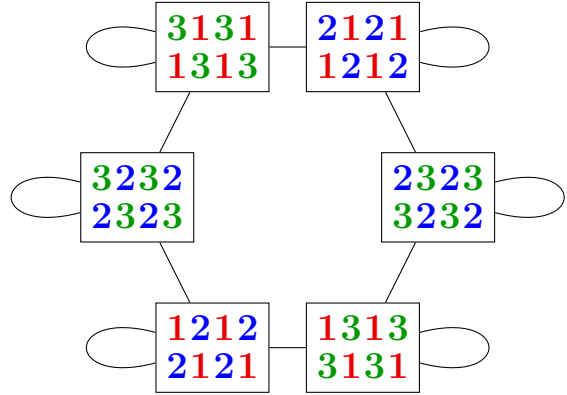
\begin{figure}
\begin{tikzpicture}[every node/.style = {draw,inner sep=1pt},
every loop/.style={}]thick,scale=.6]
\setlength{\tabcolsep}{0.4pt} 
\renewcommand{\arraystretch}{.8}
\definecolor{darkgreen}{rgb}{0,0.6,0}
\definecolor{darkred}{rgb}{.91,0,0}
\newcommand{\cR}{{\bf\textcolor{darkred}1}}
\newcommand{\cB}{{\bf\textcolor{blue}2}}
\newcommand{\cG}{{\bf\textcolor{darkgreen}3}}
\newcommand{\chk}[2]{\;\raisebox{-21pt}{{\Large\begin{tabular}{cccc}
                          #1&#2&#1&#2 \\
                          #2&#1&#2&#1 
                        \end{tabular}}}\;}
  \node (12) at (0,0) {\chk{\cR}{\cB}};
  \node (13) at (2,0) {\chk{\cR}{\cG}};
  \node (23) at (3,2) {\chk{\cB}{\cG}};
  \node (21) at (2,4) {\chk{\cB}{\cR}};
  \node (31) at (0,4) {\chk{\cG}{\cR}};
  \node (32) at (-1,2) {\chk{\cG}{\cB}};
  \draw (12) edge[loop left] (12);
  \draw (13) edge[loop right] (13);
  \draw (23) edge[loop right] (23);
  \draw (21) edge[loop right] (21);
  \draw (31) edge[loop left] (31);
  \draw (32) edge[loop left] (32);
  \draw (12) -- (13) -- (23) -- (21) -- (31) -- (32) -- (12);
\end{tikzpicture}
\caption{The graph $\mathcal{Q}$ of checkerboard colorings of~$\Z^2$.
 (A small part of each coloring is shown.)}
\label{hexagon}
\end{figure}

Conditional on the tiling, flip an independent fair coin
for each tile (e.g.\ by flipping a coin for every vertex of
$\Z^d$ and using the coin at the center of the tile). Call
a tile $T$ {\df special} if its coin is Heads but no
ancestor $\pi^m(T)$ with $1\leq m< D$ has Heads. Let $A$ be
the random set of special tiles. For any tile $T$, let
$a(T)$ be its most recent special \textit{strict\/}
ancestor, i.e.\ the tile $\pi^m(T)$ where $m\geq 1$ is the
smallest nonnegative integer for which this tile is
special.  (Such an $m$ exists a.s.)

To each special tile $T$, assign a uniformly random element
$h(T)$ of~$Q$ (again, this can be done via the center). The
idea will be that $h(T)$ will be used to color certain
descendants of~$T$.  However, the phase must be chosen
locally.  Therefore, let $h'(T)$ denote the $2$-coloring
$h(T)$ translated by $u$, where $u$ is the center of~$T$.
Thus, $h'(T)$ is either $h(T)$ or the coloring that results
from exchanging the $2$ colors, according to the parity of
$u$.

We now construct a new function $g$ from the tiles to $Q$.
The idea is that a special tile $T$ tries to force its
descendants to use $h'(T)$, succeeding if they are at least
$D$ levels below, but any special descendants get to take
over this task.

To make this precise, for any tile $T$, we choose a
shortest path in $\mathcal{Q}$ from $h'(a(a(T)))$ to
$h'(a(T))$. Here we again need to be careful with phase:
let $u$ be the center of~$a(a(T))$, and first consider the
canonical path between $h(a(a(T))$ and the translation
of~$h'(a(T))$ by $-u$, then translate all the colorings of
this path by $u$ to obtain a new path.  Let
$h'(a(a(T)))=z_0,z_1,\dots,z_\ell=h'(a(T))$ denote this
path. Now, if the distance from $T$ to $a(T)$ in
$\mathcal{F}$ is $j$, let $g(T)=z_{\min(j,\ell)}$.  We
claim that $g$ is a graph homomorphism from $\mathcal{F}$
to $\mathcal{Q}$.  Indeed, consider the parent
$T'\colonequals \pi(T)$ of $T$. If $T'$ is not special,
then $a(T')=a(T)$, so by the path construction, $g(T)$ and
$g(T')$ are neighbors in $\mathcal{Q}$.  On the other hand,
if $T'$ is special, then $a(T')$ is at distance at least
$D$ from $T'$ in $\mathcal{F}$, so $g(T')=h'(a(T'))$; but
$T$ is at distance $1$ from $a(T)=T'$ in $\mathcal{F}$, so
$g(T)$ is a neighbor of $h'(a(a(T))=h'(a(T'))$, so again
$g(T)$ and $g(T'$) are neighbours in $\mathcal{Q}$.

Now we define $X$ by assigning the checkerboard coloring $g(T)\in Q$ to all
the vertices of the tile $T$. By \cref{tiling}, each edge of the lattice
either connects two vertices in the same tile, or connects a vertex in one
tile to a vertex in that tile's parent. Since the colorings $g(T)$ and
$g(\pi(T))$ are neighbors in $\mathcal{Q}$, they are compatible, so $X$ is in
fact a 3-coloring.

It is immediate from the construction that $X$ is an automorphism-equivariant
factor of the various iid labels.
To check that it is ffiid and bound the coding radius, note
that the color $X_0$ can be determined by examining the
tile $a(\pi^D(T))$ and its descendants, where $T$ is the
tile containing $0$.  For $m\geq 1$, let $E_m$ be the event
that there is a $j$-ball containing $0$ for each of
$j=2Dm,2Dm+1,\ldots, 2D(m+1)-1$, and the coins associated
to the corresponding tiles $T_j$ are Heads for $T_{2Dm+D}$
and Tails for $T_{2Dm+D+1},\ldots,T_{2D(m+1)-1}$. On $E_m$,
tile $T_{2Dm+D}$ is special while tiles $T_{2Dm+1},\ldots,
T_{2Dm+D-1}$ tiles are not, and this is enough to determine
the coloring of tile~$T_{2Dm}$. By \cref{descendant}, tile
$T_{2Dm}$ has a descendent containing $0$. But all the
descendants of a tile are in its clump, so from
\cref{tiling} it follows that the coding radius is at most
$(\tfrac32+4) r_{2D(m+1)-1}$.  (The $4$ comes from the
construction of $V_j$ from $W_j$).  On the other hand, the
events $(E_m)_{m\geq 1}$ are independent, and their
probabilities are bounded below by some $p=p(d)>0$. Thus
$\P(R> \tfrac{11}2\times 13^{2D(m+1)})\leq (1-p)^m$, giving
the required power law bound.
\end{proof}

\section{Second moment bound}

In this section we prove \cref{power} (ii), which will
follow from a lower bound on spatial correlations that
holds for any stationary $3$-coloring of~$\Z^2$. The key to
the proof is that there is a height function associated to
the $3$-colorings.  If correlations were to decay too fast
then the height changes around a large contour would not
cancel.

We need the following simple lemma, the proof of which is deferred to the end
of the section.  A process $Z$ on $\Z$ is called {\df right-tail-trivial} if
every event in $\mathcal{T}_+\colonequals
\bigcap_{n\in\Z}\,\sigma(Z_n,Z_{n+1},\ldots)$ has probability zero or one.

\begin{lemma}\label{variance}
If $(Z_i)_{i\in \Z}$ is a $\pm 1$-valued stationary
right-tail-trivial process,  either it is a.s.\
deterministic or $\limsup_{n\to\infty}\Var\sum_{i=1}^n Z_i
= \infty$.
\end{lemma}

Let $X$ be a $3$-coloring of~$\Z^2$.  We will prove a lower
bound on spatial correlations involving pairs of edges. Let
$u,v\in\Z^2$ be neighboring vertices. Since $X$ is a
coloring, $X_v-X_u\equiv \pm1 \pmod 3$. Therefore define
$h(u,v)\in\{-1,+1\}$ by $h(u,v)\equiv X_v-X_u \pmod 3$. Now
define
\begin{multline*}
\rho(r)\colonequals \sup\Bigl\{\Cov\bigl[h(u_1,v_1),h(u_2,v_2)\bigr] : \\
\;\|u_1-v_1\|_1=\|u_2-v_2\|_1=1,\;\;\|u_1+v_1-u_2-v_2\|_1\ge 2r\Bigr\},
\end{multline*}
Note that $\rho$ is nonnegative (since interchanging $u_1$
and $v_1$ reverses the sign of the covariance), and
non-increasing.

\begin{proposition}[Correlations]\label{correlations}
Let $X$ be a stationary $3$-coloring of~$\Z^2$, and suppose
that its restriction $(X_{(i,0)})_{i\in\Z}$ to the axis is
right-tail-trivial.  Then with $\rho$ defined as above,
\[\sum_{r=1}^\infty r \rho(r)=\infty\,.\]
\end{proposition}

The key point is that the function $h$ defined above can be
interpreted as the difference along an edge of a height
function.  (See, e.g.\ \cite{baxter},\cite{galvin} for
background.)  Indeed, suppose $w_0,\ldots,w_3$ are the
vertices of a unit square of~$\Z^2$ in counterclockwise
order, and write $w_4=w_0$. Then $\sum_{j=0}^3
h(w_{j},w_{j+1})=0$ (since the sum lies in
$\{0,\pm2,\pm4\}$ but equals $0$ modulo $3$). Therefore,
for arbitrary vertices $u,v\in\Z^2$ we can define
$h(u,v)\colonequals \sum_{j=0}^{m-1} h(w_j,w_{j+1})$ where
$u=w_0,w_1,\ldots,w_m=v$ is any path from $u$ to $v$; it
follows from the above observation that this sum does not
depend on the choice of path.

\begin{proof}[Proof of Proposition~\ref{correlations}]
Let $X$ be a $3$-coloring with the given properties, and suppose for a
contradiction that $\sum_{r=1}^\infty r \rho(r)=C<\infty$.

Write $v_j\colonequals (j,0)$, and let $n\geq 1$.  We will bound the
variance of~$h(v_0,v_n)$ by expressing it in two different
ways. Summing along the axis gives:
\[
h(v_0,v_n) =
\sum_{j=1}^n h(v_{j-1},v_j)\,,
\]
while by summing around three sides of a square:
\begin{align*}
h(v_0,v_n) = &\sum_{j=1}^n  h\bigl((0,j-1),(0,j)\bigr)+ \sum_{j=1}^n
h\bigl((j-1,n),(j,n)\bigr)\\[-3pt]&
+ \sum_{j=1}^n h\bigl((n,n-j+1),(n,n-j)\bigr).
\end{align*}
Thus, we may compute $\Var
h(v_0,v_n)=\Cov[h(v_0,v_n),h(v_0,v_n)]$ as the covariance
of the two representations. This gives
\begin{equation}\label{var-bound}
\Var h(v_0,v_n) \leq 2\sum_{r=1}^{2n} r\rho(r) +
\sum_{r=n}^{2n} 2n\rho(r)
  \leq 4C.
\end{equation}

Now let $Z_i\colonequals h(v_{i},v_{i+1})$.  The assumption on the
coloring $X$ implies that the process $Z=(Z_i)_{i\in\Z}$ is
right-tail-trivial, so using \eqref{var-bound} and applying
Lemma~\ref{variance} shows that $Z$ is deterministic, which
is to say that either $Z_i=1$ for all $i$ a.s.\ or $Z_i=
-1$ for all $i$ a.s.  Without loss of generality consider
the former case. Then the coloring $(X_{(i,0)})_{i\in\Z}$
restricted to the axis is supported on the set of three
$3$-periodic colorings of the form $\cdots123123\cdots$.
Stationarity implies that these three colorings must each
have probability $1/3$, but the resulting process is not
right-tail-trivial.
\end{proof}

\begin{proof}[Proof of \cref{power} (ii)]
Let $X$ be an ffiid $3$-coloring of~$\Z^d$ with $d\geq 2$, and suppose for a
contradiction that the coding radius satisfies $\E R^2<\infty$.  Similarly to
the proof of \cref{tower} (ii), we may assume without loss of generality that
$d=2$, since restricting an ffiid process to the plane
$\Z^2\times\{0\}^{d-2}$ gives another ffiid process, and does not increase
the coding radius.  We claim also that $X$ restricted to $\Z\times\{0\}$ is
right-tail-trivial as required for Proposition~\ref{correlations}.  This
is a consequence of a fact from ergodic theory: any process that is a factor
of an iid process and takes values in $A^\Z$, where $A$ is a finite set, is
right-tail-trivial; see e.g.\ \cite[Thm.~1 on p.~283, Ex.~1 on p.~280, and
Def.~3 on p.~181]{cornfeld}.  Alternatively, an elementary argument shows
that an ffiid process with $\E R^2<\infty$ satisfies a stronger
tail-triviality condition; we explain this at the end of the section ---
specifically we use \cref{tail-trivial} with $d=2$.

We now bound $\rho(r)$.  Let $\|u_1+v_1-u_2-v_2\|_1\geq 2r$
and $\|u_i-v_i\|_1=1$ for $i=1,2$.  Write $H_i\colonequals
h(u_i,v_i)-\E h(u_i,v_i)$, so that $\E H_i=0$ and
$|H_i|\leq 2$.  Recall that $R_v$ denotes the coding radius
at vertex~$v$, and define the event $E_i\colonequals
\{R_{u_i}\vee R_{v_i}> r/2-1\}$. Thus, the random variables
$H_1 \ind_{\overline E_1}$ and $H_2\ind_{\overline E_2}$
are independent, since they are functions of disjoint sets
of iid variables. Writing $\eps=\eps(r)\colonequals \P(R>
r/2-1)$, note that $\P (E_i) \leq 2\eps$, and also $\E(H_i
\ind_{\overline E_i})= -\E(H_i\ind_{E_i})\leq 4\eps$.
Therefore,
\begin{align*}
\E (H_1H_2) &=
\E\bigl(H_1 H_2 \ind_{E_1\cup E_2}\bigr)
+\E\bigl(H_1 H_2 \ind_{\overline E_1} \ind_{ \overline E_2}\bigr) \\
&\leq 4\,\P(E_1\cup E_2) + \E(H_1 \ind_{\overline E_1})
\,\E(H_2 \ind_{\overline E_2}) \\
&\leq 16\eps +16\eps^2\leq 32\eps,
\end{align*}
and thus $\rho(r)\leq 32\,\P(R> r/2-1)$.
Proposition~\ref{correlations} gives $\sum_{r} r \rho(r)
=\infty$, so $\sum_r r\,\P(R>r/2-1)=\infty$, which implies
$\E R^2=\infty$.
\end{proof}

We conclude the section by giving the proof of
Lemma~\ref{variance}, and also the elementary argument for
tail-triviality mentioned above.

\begin{proof}[Proof of Lemma~\ref{variance}]
Let $(Z_i)_{i\in\Z}$ be stationary and $\pm 1$-valued, and
suppose $\Var \sum_{i=1}^n Z_i\leq C^2$ for all $n$.  We
will deduce that $Z$ is deterministic.

Let $\mathcal{F}_j$ be the $\sigma$-field generated by
$Z_j,Z_{j+1},\ldots$, and consider the space of random
variables $L^2(\mathcal{F}_j)$, with the norm $\|X\|_2\colonequals
(\E X^2)^{1/2}$. Write $\mu\colonequals \E Z_0$ and $S_j^k\colonequals \sum_{i=j}^{k-1}
(Z_i-\mu)$, so that $\|S_j^k\|_2^2\leq C^2$ and in
particular $S_j^k\in L^2(\mathcal{F}_j)$. Now define
$\phi_j: L^2(\mathcal{F}_j)\to [0,\infty)$ by
\[
\phi_j(X)\colonequals \limsup_{n\to\infty} \E{\bigl( X+S_j^n\bigr)^2}\,.
\]
We will prove that $\phi_j$ has a unique global minimizer
in $L^2(\mathcal{F}_j)$.

First note that by the Cauchy-Schwarz inequality and the
uniform bound on $\|S_j^k\|_2^2$, the function $\phi_j$
satisfies the bounds
\[
\|X\|_2^2-2C\,\|X\|_2\leq \phi_j(X)\leq \|X\|_2^2+2C\,\|X\|_2+C^2\,,
\]
and therefore $\phi_j(X)<\infty$ for all $X\in
L^2(\mathcal{F}_j)$, while $\phi_j(X)\to\infty$ as
$\|X\|_2\to\infty$.

We next claim that $\phi_j$ is continuous and strictly
convex. To check continuity, let $X,Y\in
L^2(\mathcal{F}_j)$ satisfy $\|X-Y\|_2=\eps$.  Writing
$Y+S_j^n=(X+S_j^n)+(Y-X)$ and using Cauchy-Schwarz again,
\begin{equation}\label{cont-bound}
\phi_j(Y)\leq \phi_j(X) +
2\eps \,\phi_j(X)^{1/2} +\eps^2.
\end{equation}
Applying \eqref{cont-bound} in both directions, using
\eqref{cont-bound} again to bound $\phi_j(Y)$ in terms
of~$\phi_j(X)$ on the right side, and simplifying, we
obtain
\[\bigl|\phi_j(Y)-\phi_j(X)\bigr|\leq 2\eps\,
\phi_j(X)^{1/2}+3\eps^2\,,\] from which continuity follows.

To check that $\phi_j$ is strictly convex, observe that for
$X,Y\in L^2(\mathcal{F}_j)$,
\[
\begin{aligned}
\phi_j\Bigl(\frac{X+Y}2\Bigr)
&
=
\limsup_{n\to\infty} \E\left[\left(\frac{ (X+S_j^n)+(Y+S_j^n)}2\right)^2\right]
\\&
=
\limsup_{n\to\infty} \E\left[\frac{ (X+S_j^n)^2+(Y+S_j^n)^2}2 -\frac{(X-Y)^2}4\right]
\\&
\le \frac{\phi_j(X)+\phi_j(Y)}2 - \frac{\|X-Y\|_2^2}4\,.
\end{aligned}
\]

It now follows (see e.g.\ \cite[Theorem 2.11, Remarks 2.12,
2.13]{Barbu-Precupanu}) that $\phi_j$ has a unique
mimimizer. Let $X_j\in L^2(\mathcal{F}_j)$ minimize
$\phi_j$. For $j<k<n$ we have $S_j^n=S_j^k+S_k^n$, and
hence $\phi_k(X+S_j^k)=\phi_j(X)$. Therefore,
\[
X_j+S_j^k=X_k\,.
\]
By construction, $(X_i)_{i\in\Z}$ is stationary. Since
$X_j=X_{j+1}-S_j^{j+1}=X_{j+1}-Z_j+\mu$ we have
\begin{equation}\label{XX}
X_j\equiv X_{j+1}+1+\mu \mod 2.
\end{equation}
Thus, $X_j \mmod 2 \in L^2(\mathcal{F}_{j+1})$, and by
iterating we see that $X_j \mmod 2$ is in the right tail of
$(Z_i)_{i\in\Z}$. Therefore, $X_j \mmod 2$ is an a.s.\
constant for each $j$. Since $(X_i)_{i\in\Z}$ is
stationary, we have $X_0\equiv X_1 \mmod 2$ a.s.  Hence
\eqref{XX} gives $\mu\equiv 1\mmod 2$, i.e.\
$\mu\in\{-1,+1\}$, so $Z_0$ is a.s.\ deterministic.
\end{proof}

A process $X$ on $\Z^d$ is called {\df fully tail-trivial}
if every event in $\mathcal{T}(X)\colonequals
\bigcap_{r\geq 0}\,\sigma(X_v:v\not\in B(r))$ has
probability zero or one. Of course, the restriction of a
fully tail-trivial process to the axis is also fully
tail-trivial, and therefore right-tail-trivial.  Hence, the
following lemma suffices for our needs in the proof of
\cref{power} (ii) above.

\begin{lemma} \label{tail-trivial}
Suppose $X$ is an ffiid process on $\Z^d$.  If the coding
radius~$R$ satisfies $\E R^d<\infty$ then $X$ is fully
tail-trivial.
\end{lemma}

\begin{proof}
Let $X$ be a finitary factor of the iid process $Y$ with
coding radius satisfying $\E R^d<\infty$. For $u,v\in\Z^d$
we write $u\hookrightarrow v$ for the event $\{|u-v|\leq
R_v\}$, i.e.\ the event that $u$ is within the ball that
must be examined to determine $X_v$.

For positive integers $n<N$, define
\[E_{n,N}\colonequals \bigl\{\exists\; u\in B(n)
\text{ and }v\notin B(N)\text{ s.t.\ }u\hookrightarrow
v\bigr\}\,.\]
We claim that for any $n$ we have $\P(E_{n,N})\to 0$ as
$N\to\infty$. Indeed, by translation-invar\-iance and the
assumption on $R$,
\[\sum_{v\in\Z^d} \P(0\hookrightarrow v)=\sum_{v\in\Z^d} \P(-v\hookrightarrow 0)
=\sum_{v\in\Z^d} \P(v\hookrightarrow 0)=\E |B(R)|<\infty\,.\] (This is an
instance of the ``mass-transport principle'' --- see
\cite{blps} for background.) Hence for any $n$,
\[\sum_{\substack{u\in B(n) ,\; v\in\Z^d}} \P(u\hookrightarrow v)<\infty\,,\]
and thus
\[\P(E_{n,N})\leq \sum_{\substack{u\in B(n) ,\; v\notin B(N)}} \P(u\hookrightarrow v)\to
0\quad\text{as }N\to\infty\,,\] as claimed.

Now, fix $\eps>0$ and a tail event $A\in\mathcal{T}(X)$.
Since $X$ is a function of~$Y$ we have $A\in\sigma(Y)$, so
we can find an approximating cylinder event: there exist
$n$ and $A'\in\sigma(Y_v:v\in B(n))$ such that
$\P(A\triangle A')<\eps$. Let $A''=A\setminus E_{n,N}$. By
the above claim, for $N$ large enough we have
$\P(A\triangle A'')<\eps$.  On the other hand, since $A$ is
a tail event, $A\in\sigma(X_v:v\not\in B(N))$, and so by
the definition of~$E_{n,N}$ we deduce
$A''\in\sigma(Y_v:v\notin B(n))$. Thus $A''$ and $A'$ are
independent.  Since $\eps$ was arbitrary this implies that
$A$ is independent of itself, i.e.\ $\P(A)\in\{0,1\}$.
\end{proof}

\section{Finitely dependent coloring}

In this section we prove two results on $k$-dependent
coloring.  See \cite{hl} for more on this topic.

\begin{proof}[Proof of \cref{fin-dep}]\sloppypar
Suppose $(X_i)_{i\in\Z}$ is a translation-invariant
$1$-dependent coloring.  Let $\mathcal{F}_i\colonequals
\sigma(\ldots,X_{i-1},X_i)$, and define the random variable
\[Y_i\colonequals \P(X_{i+1}=1\mid \F_i)\,.\]
Let $m=\esssup Y_1;$ then since $X_2\neq 1$ on the event $X_1=1$ we have a.s.
\[Y_1 \leq m\, \ind[X_1\neq 1]\,.\]
Since $X_2$ is independent of~$\mathcal{F}_0$, we deduce
\begin{multline*}
\P(X_2=1)=\P(X_2=1 \mid\mathcal{F}_0)= \E(Y_1\mid
\mathcal{F}_0)\\ \leq m\, \P(X_1\neq 1\mid \mathcal{F}_0)=m\,(1-Y_0)\,,
\end{multline*}
 and since this holds a.s.\ we deduce
\[\P(X_2=1)\leq m\,(1-m)\leq 1/4\,.\]
Similarly we have $\P(X_2=k)\leq 1/4$ for each color
$k=1,\ldots,q$, so $q\geq 4$.
\end{proof}

Finally, we note the following consequence of the results
of the previous section.
\begin{corollary}\label{kdep-zd}
Let $d\geq 2$ and $k\geq 1$.  There exists no stationary $k$-dependent
$3$-coloring of~$\Z^d$.
\end{corollary}

\begin{proof}
By restricting to a plane, it is enough to prove the $d=2$ case. We use
\cref{correlations}.  It is elementary to check that the restriction of a
$k$-dependent process to the axis is right-tail-trivial, and that the
correlation function $\rho(r)$ is zero for $r>k+2$.
\end{proof}
In contrast with \cref{kdep-zd}, in dimension $d=1$, a
stationary $2$-dependent $3$-coloring was constructed in
\cite{hl}.

\section*{Open Problems}

\begin{ilist}
\item What is the largest $\alpha$ for which there exists
    an ffiid $3$-coloring of~$\Z^d$ whose coding radius has
    finite $\alpha$-moment, for each $d\geq 2$?  (Our
    results show that it is at most $2$, and at least some
   small positive number.)
\item Does there exist, for some $d\geq 2$, a shift of
    finite type $S$ on $\Z^d$ that contains no constant
    configuration, but that admits some ffiid process $X$ with
    $X\in S$ a.s.\ whose
    coding radius tail decays strictly faster than a
    tower function?
\item Does there exist, for some $d\geq 2$, a shift of
    finite type $S$ that admits an ffiid $X$
    with all moments of the coding radius finite, but
    that admits no $X$ with tower function decay?  (For example, can
    the optimal tail decay be exponential?)
\item Does there exist, for some $d\geq 2$, a shift of
    finite type $S$ that admits an ffiid $X$, but
    such that all moments of the coding radius
    are infinite for every such $X$?
\end{ilist}
Our results imply negative answers to (ii),(iii),(iv) in dimension $d=1$.

\section*{Acknowledgements}
We thank Itai Benjamini, Jeff Steif, and Benjamin Weiss for valuable
discussions.  We thank an anonymous referee for helpful
suggestions.
 \enlargethispage*{1cm}

\newcommand{\MRhref}[2]{\href{http://www.ams.org/mathscinet-getitem?mr=#1}{MR#1}}
\def\@rst #1 #2other{#1}
\renewcommand\MR[1]{\relax\ifhmode\unskip\spacefactor3000 \space\fi
  \MRhref{\expandafter\@rst #1 other}{#1}}
\renewcommand{\arXiv}[1]{\href{http://arxiv.org/abs/#1}{arXiv:#1}}
\renewcommand{\arxiv}[1]{\href{http://arxiv.org/abs/#1}{#1}}

\bibliographystyle{hmralphaabbrv}
\bibliography{col}
\vspace{5mm}

\end{document}